\documentclass[11pt]{article}
\usepackage[latin1]{inputenc}
\usepackage[T1]{fontenc}
\usepackage{authblk} 
\usepackage{amsmath}
\usepackage{amsfonts}
\usepackage{amssymb}
\usepackage[normalem]{ulem} 
\usepackage{amsmath,xcolor,ulem}
\usepackage{amsfonts}
\usepackage{amssymb}
\usepackage{amsthm}
\usepackage{graphicx}
\usepackage{marginnote}
\usepackage{subcaption}

\usepackage[colorlinks=true, allcolors=blue]{hyperref}
\usepackage[top=2.5cm,bottom=2.5cm,left=2.5cm,right=2.5cm]{geometry}
\usepackage{float}

\usepackage[algoruled]{algorithm2e}
\usepackage[font=footnotesize,width=\linewidth]{caption} 
\usepackage{graphicx}
\usepackage{amssymb}
\usepackage{amsthm}
\usepackage{amsmath}
\usepackage{tikz} 
\usepackage{lineno}
\usepackage{hyperref}
\usepackage{verbatim}

\newtheorem{theorem}{Theorem}[section]
\newtheorem{lemma}[theorem]{Lemma}
\theoremstyle{definition}

\theoremstyle{remark}
\newtheorem{remark}[theorem]{Remark}

\numberwithin{equation}{section}

\DeclareMathOperator{\interior}{int}
\DeclareMathOperator{\trace}{tr}
\DeclareMathOperator{\err}{err}

\usepackage{graphicx, pdflscape}
\usepackage{amsmath, hyperref}
 \usepackage{amssymb} 
\usepackage{relsize, float, lineno}
 
\usepackage{tikz} 
 
 \usepackage{xcolor} 
\usepackage{subcaption}
\allowdisplaybreaks

\title{A Generalized Hacker Dynamics Model with Nonlinear Incidence: Analysis and Positivity-Preserving Numerical Simulation}

\author[1]{Manh Tuan Hoang\footnote{\href{mailto:tuanhm16@fe.edu.vn}{tuanhm16@fe.edu.vn}}}
\author[2,3]{Hoai Thu Pham \footnote{\href{mailto: phamthuhvan@gmail.com}{phamthuhvan@gmail.com}}}
\author[4]{Matthias Ehrhardt\footnote{ \href{mailto:ehrhardt@uni-wuppertal.de}{ehrhardt@uni-wuppertal.de} (corresponding author)} }
\affil[1]{Department of Mathematics, FPT University, Hoa Lac Hi-Tech Park, Km29 Thang Long Blvd, Hanoi, Viet Nam}
\affil[2]{People's Security Academy, Hanoi, Viet Nam}
\affil[3]{Graduate University of Science and Technology, Vietnam Academy of Science and Technology, Hanoi, Viet Nam}
\affil[4]{IMACM, School of Mathematics and Natural Sciences, University of Wuppertal, Germany}
\begin{document}
\maketitle

\begin{tikzpicture}[remember picture,overlay]
\node[anchor=north east,inner sep=20pt] at (current page.north east)
	{\includegraphics[scale=0.2]{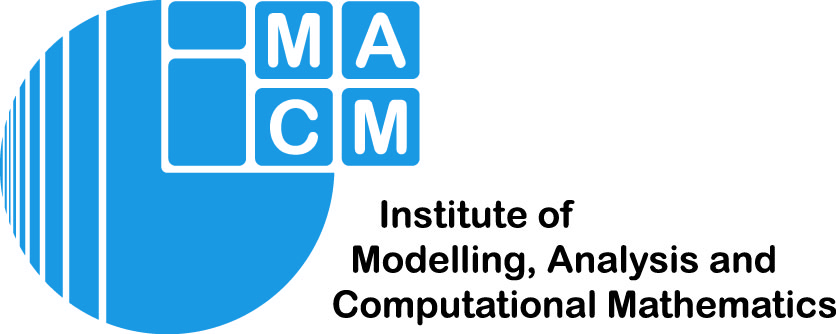}};
\end{tikzpicture}

\begin{abstract}
We propose and analyze a generalized compartment model for hacker dynamics in cybersecurity systems. This model is an extension of a recently introduced framework, replacing the bilinear interaction term with a broad class of nonlinear incidence functions. This provides greater modeling flexibility and allows for the description of a wider range of cyber-propagation and information-spreading processes.

First, we investigate the qualitative dynamics of the model. We establish the positivity and boundedness of solutions, derive the basic reproduction number, and characterize the existence of hacker-free and hacker-present equilibria. We obtain local and global asymptotic stability results, yielding a complete description of the global dynamics in terms of the basic reproduction number.

Next, we develop a second-order, positivity-preserving, nonstandard finite difference (NSFD) scheme for numerical simulations. Unlike many existing NSFD approaches, which are typically first-order accurate, the proposed method achieves second-order convergence while preserving positivity for arbitrary step sizes. Furthermore, the proposed method reproduces the asymptotic stability properties of the continuous model, making it suitable for long-time simulations.

Numerical experiments validate the theoretical results and demonstrate the superior accuracy and qualitative performance of the proposed scheme compared to several first-order methods.

\end{abstract}

\begin{minipage}{0.9\linewidth}
 \footnotesize
\textbf{AMS classification:} 34C60, 37N99, 65L05, 65Z05.

\medskip\noindent
\textbf{Keywords:} 
Compartment epidemic models; Hacker dynamics; Cybersecurity systems; Nonstandard finite difference; Positivity-preserving; Second-order

\end{minipage}

\section{Introduction}\label{intro}
Compartment models have several useful applications in both theory and practice, particularly in mathematical epidemiology \cite{Allen, Angstmann, Brauer, Kermack-McKendrick, Kermack-McKendrick1, Kermack-McKendrick2, Martcheva}. 
One of the earliest compartmental epidemic models, often used to introduce the mathematical modeling of epidemiology, was proposed by Kermack and McKendrick in 1927 \cite{Kermack-McKendrick} (see, also, \cite{Martcheva}). 
This model is represented by
\begin{equation}\label{eq:SIR}
\begin{split}
\dfrac{dS(t)}{dt} &= -\beta S(t) I(t),\\
\dfrac{dI(t)}{dt} &= \beta S(t) I(t) - \alpha I(t),\\
\dfrac{dR(t)}{dt} &= \alpha R(t),
\end{split}
\end{equation}
where $S(t)$, $I(t)$ and $R(t)$ denote the number of individuals in the susceptible, infected and recovered compartments, respectively; $\beta$ and $\alpha$ are called the transmission and recovery rate constants, respectively.

Over the past several decades, epidemiological compartmental models have been widely used to study information and malware propagation. These models have become an effective approach with several useful applications in this field 
(see, for example, \cite{Daley, delRey, Dietz, Ding, Goffman, Kawachi, Murray, Piqueira1, Piqueira2, Piqueira3, Piqueira4, Zhu} and references therein).  
Early studies analyzed and discussed the strong similarity between epidemic spreading and the propagation of information and malware \cite{Daley, Dietz, Goffman, Murray}. 
Subsequently, epidemiological models were developed to study issues related to the propagation of information and malware. 
Extended versions of the classical epidemic models \cite{Kermack-McKendrick, Kermack-McKendrick1, Kermack-McKendrick2} have been formulated to study the propagation of information and malware on the Internet 
 \cite{Kawachi, Piqueira1, Piqueira2, Piqueira3, Piqueira4, Zhu}.

In a recent survey \cite{HoangMatthias}, Hoang and Ehrhardt presented developments in mathematical modeling and applications of differential equations to infectious diseases. 
They focused on mathematical models, qualitative analysis, numerical methods, and practical applications.
They also mentioned and discussed the application of epidemiological models to the study of information and malware propagation.

In this work, we revisit a compartment model of hacker behavior in cybersecurity systems, which was proposed in a recent work \cite{Hassouni}.
The model analyzes the influence of black hat hackers on Internet users.
In the mathematical model, Internet users are classified into three classes defined as follows:
\begin{itemize}
\item Class $C$: This class consists of potential Internet users influenced by positive and negative factors related to Internet use.
\item Class $H_n$ (negative): This class represents \textit{black hat hackers} involved in illegal Internet activities, including stealing personal information and bank account data, conducting cyber espionage, and attacking websites.
\item Class $H_b$ (benefit): These are \textit{white hat hackers} who contribute positively to cyberspace by detecting vulnerabilities in websites and cooperating with the relevant authorities to enhance website security.
\end{itemize}
Based on some hypotheses inspired by real-world applications (see \cite[Fig.~1]{Hassouni}), 
the following system of ordinary differential equations (ODEs) was proposed to model the hacker dynamics
\begin{equation}\label{eq:1}
\begin{split}
 \dfrac{d C(t)}{d t} &= \Lambda-\beta C(t) H_n(t)-\alpha C(t)+\gamma H_n(t)-\mu C(t), \quad C(0) > 0 \\
 \dfrac{d H_n(t)}{d t} &= \beta C(t) H_n(t)-\gamma H_n(t)-\theta H_n(t)-\mu H_n(t),\quad H_n(0) \ge 0\\
 \dfrac{d H_b(t)}{d t} &= \alpha C(t)+\theta H_n(t)-\mu H_b(t), \quad H_b(0) \ge 0,
\end{split}
\end{equation}
where all the parameters are assumed to be positive. Their meanings were described in \cite[Table~1]{Hassouni}.

A rigorous mathematical analysis of \eqref{eq:1} was performed in \cite{Hassouni}. 
Specifically, the invariant set, the basic reproduction number $\mathcal{R}_0$, the set of equilibrium points, and their local and global asymptotic stability were determined. 
Additionally, a sensitivity analysis of $\mathcal{R}_0$ and an optimal control problem were investigated.

As we recall from \cite{Martcheva}, the function $\lambda(t) := \beta I(t)$ in \eqref{eq:SIR} is called the force of infection. 
There are several types of incidence depending on the assumptions about the form of the force of infection. 
One commonly used form is the \textit{mass action incidence} $\beta SI$. 
This interaction term increases linearly with the number of infected individuals. 
However, Capasso and Serio \cite{Capasso} pointed out that the mass action incidence may not be realistic in real-world applications. 
For this reason, several studies (see, for instance, \cite{ConnellMcCluskey, Sun, Tian, Xu, Zhang} and references therein), have considered nonlinear incidence functions to replace the mass action incidence. 
More precisely, the bilinear function $f(S, I) := \beta SI$ is replaced with nonlinear functions $f(S, I)$ that satisfy:
\begin{itemize}
\item[\textbf{(P1)}] $f\colon\mathbb{R}^2_+ \to \mathbb{R}_+$  is a differentiable function with $f(S, 0) = f(0, I) = 0$ and $f(S, I) > 0$ for all $S, I > 0$.
\item[\textbf{(P2)}] $\dfrac{\partial f(S, I)}{\partial S} > 0$ for all $S \ge0$ and $I > 0$.
\item[\textbf{(P3)}] $\dfrac{\partial f(S, I)}{\partial I} \ge0$ for all $S, I \ge0$.
\item[\textbf{(P4)}] $I\dfrac{\partial f(S, I)}{\partial I} - f(S, I) \le0$ for all $S, I \ge0$.
\item[\textbf{(P5)}] There exists $\eta > 0$ such that $f(S, I) \le\eta I$ for all $S, I \ge0$.
\end{itemize}

It is worth noting that the class of functions $f$ with properties  \textbf{(P1)}--\textbf{(P5)} includes many well-known incidence functions, for example,
\begin{equation*}\label{eq:Nonlinear1}
   f(S, I) = \dfrac{\beta SI}{1 + \kappa I}, \qquad f(S, I) = \dfrac{\beta SI}{1 + \kappa_1 S + \kappa_2 I},
\end{equation*}
correspond to the saturated incidence \cite{Capasso} and the Beddington-DeAngelis response \cite{Beddington, DeAngelis, Zhang}.
Nonlinear functions with properties \textbf{(P1)}--\textbf{(P5)} were used in \cite{ConnellMcCluskey, Sun}, whereas those satisfying properties \textbf{(P1)}--\textbf{(P4)} were considered in \cite{Tian}. 
Also, a generalized hepatitis B epidemic model using a nonlinear incidence rate was proposed and analyzed in \cite{Hoang2023}. 

Inspired by the aforementioned studies that used nonlinear incidence functions, we generalize equation \eqref{eq:1} by replacing the bilinear function  $C H_n$ with a function $f(C, H_n)$ that satisfies \textbf{(P1)}--\textbf{(P4)}. 
More precisely, we propose the following compartment model:
\begin{equation}\label{eq:1G}
\begin{split}
 \dfrac{dC}{dt} &= \Lambda - f(C, H_n) -\alpha C + \gamma H_n -\mu C, \\
 \dfrac{dH_n}{dt} &= f(C, H_n) - \gamma H_n - \theta H_n - \mu H_n,\\
 \dfrac{d H_b}{dt} &= \alpha C + \theta H_n - \mu H_b.
\end{split}
\end{equation}

There are some practical considerations that justify replacing the bilinear incidence term $\beta CH_n$ with the nonlinear incidence $f(C,H_n)$. 
In real-world cyber environments, the influence of black-hat hackers does not always increase linearly with population size due to limited user interaction, increasing cybersecurity awareness, information overload, and network constraints. 
Furthermore, when the number of black-hat hackers is large enough, their overall influence may gradually saturate because the number of susceptible users is finite, and protective mechanisms become more effective. 
Therefore, a nonlinear incidence function provides a more realistic description of the propagation dynamics of malware and cyber activities.
Conversely, introducing nonlinear incidence rates makes the generalized model more flexible, which is useful for parameter estimation problems with real data.

Clearly, \eqref{eq:1G} is a generalization that is also more flexible than \eqref{eq:1}. 
Consequently, it can generate a broader spectrum of realistic scenarios.  It can also be useful for modeling the spread of information on social networks. 
Readers are referred to \cite{Bhattacharya, Daley, Dietz, Goffman, Piqueira4} for compartmental epidemiological systems for modeling information spreading.
From the perspective of information spreading, the model can be interpreted as follows:
\begin{itemize}
\item The compartment $C$ can be interpreted as a group of neutral internet users, or users who do not have a clear opinion about a given piece of information. 
When information reaches individuals in this compartment, they may believe it and participate in disseminating it, thereby becoming members of compartment $H_n$. 
Alternatively, they may recognize the information's falsehood, oppose or verify it, and consequently move to compartment $H_b$.
\item The compartment $H_n$ represents individuals who spread positive or negative information. 
After some time, individuals in this compartment may stop spreading information and return to the neutral state $C$, or be convinced by reliable information and move to compartment $H_b$.
\item The compartment $H_b$ includes individuals who disseminate reliable information and participate in fact-checking and information verification to reduce the spread of misinformation in online communities.
\end{itemize}
Note that the parameter $\mu$ now represents the rate at which individuals leave the system from each compartment, 
while the parameter  $\Lambda$ represents the recruitment rate of new users into the online system.
%

First, we examine the positivity and boundedness of the solutions, compute the basic reproduction number, and determine the hacker-free (HFE) and hacker-present (HPE) equilibrium points. 
Then, we establish their local and \textit{global asymptotic stability} (GAS) properties. 
Note that stability results for cascade-structured systems \cite{Seibert} are applied to examine the GAS problem. 
A recent study \cite{Hoang2026} used this approach to establish the GAS of a generalized epidemiological model for malware propagation.
Next, we analyze the GAS of the HFE using either the connection between the basic reproduction number and the disease-free equilibrium in epidemic models \cite{Castillo-Chavez} 
or a quadratic Lyapunov function.
We examine the GAS of the HPE based on the Poincar\'e-Bendixson theorem combined with the Bendixson-Dulac criterion for two-dimensional dynamical systems \cite{Allen, Martcheva}.
As a result, the global dynamics of the generalized compartment model are fully characterized.

Second, for the numerical simulation of \eqref{eq:1G}, we construct a positivity-preserving \textit{nonstandard finite difference} (NSFD) scheme \cite{Mickens1, Mickens2, Mickens3, Mickens4, Mickens5}. 
The positivity of solutions is a fundamental property of mathematical models that describe real-world systems \cite{Allen, Brauer, Martcheva, Smith},
so it should be preserved by numerical methods \cite{Ascher, Horvath, Stuart}.
The NSFD approach, developed by Mickens \cite{Mickens1, Mickens2, Mickens3, Mickens4, Mickens5},
provides an effective framework for constructing numerical schemes that preserve essential qualitative properties and other important features. 
Surveys on the development of NSFD methods and their applications can be found in \cite{HoangMatthias, Patidar1, Patidar2}. 
Unlike the commonly used approaches in \cite{Cresson, Dimitrov3, Wood0, Wood1} that generally yield first-order NSFD schemes, 
we adopt the approach recently proposed in \cite{HoangMatthias2026} to derive a second-order, positivity-preserving NSFD scheme. 
This scheme does not impose the condition that the right-hand side function vanish at any time step. 
Analysis and numerical simulations show that the derived NSFD scheme significantly improves the accuracy of well-known first-order NSFD schemes resulting from \cite{Cresson, Dimitrov3, Wood0, Wood1}.
Moreover, the scheme preserves the positivity and asymptotic stability properties of the continuous-time model. 
Thus, the second-order NSFD scheme is particularly useful for simulating the dynamics of continuous-time models over long time periods.

Lastly, we conducted several numerical experiments to validate the theoretical findings and demonstrate the advantages of the implemented second-order NSFD scheme. 
The numerical results provide strong evidence in support of the theoretical assertions.
This work is organized as follows.
Section~\ref{Sec2} investigates the dynamical properties of the proposed model \eqref{eq:1G}. 
Section~\ref{Sec3} constructs and analyzes the second-order positivity-preserving NSFD scheme.
Section~\ref{Sec4}reports the results of the numerical experiments. 
The final section presents concluding remarks and discussions.

\section{Dynamical Analysis of the Generalized Model}\label{Sec2}
In this section, we analyze the dynamical properties of the proposed model \eqref{eq:1G}.
\subsection{Basic properties}\label{Subsec2}
First, we analyze some basic properties of the model \eqref{eq:1G}.
\begin{theorem}\label{Theorem1}
The following assertions hold for the compartment model \eqref{eq:1G}:
\begin{itemize}
\item[(i)] It admits the positive orthant of $\mathbb{R}^3$, 
$\mathbb{R}_+^3 = \{(C,\,H_n,\,H_b) \in \mathbb{R}^3|C, H_n, H_b \ge0\}$,
as a positive invariant set. 
\item[(ii)] If $C(0)\ge0, H_n(0) > 0, H_b(0) \ge0$, then $C(t), H_n(t), H_b(t)>0$ for $t>0$.
\item[(iii)] The total population $N(t) = C(t) + H_n(t) + H_b(t)$ ($t \ge0$) monotonically converges to $N^* = \frac{\Lambda}{\mu}$.
\item[(iv)] All the solutions satisfy the following estimates
\begin{equation}\label{eq:9}
\begin{split}
&\limsup_{t \to \infty}\big(C(t) + H_n(t)\big) \le \dfrac{\Lambda}{\min\{\alpha + \mu,\,\theta + \mu\}},\\
&\liminf_{t \to \infty}\big(C(t) + H_n(t)\big) \ge \dfrac{\Lambda}{\max\{\alpha + \mu,\,\theta + \mu\}},\\
&\limsup_{t \to \infty}H_b(t) \le \dfrac{{\max\{\alpha,\,\theta\}}}{\mu}\dfrac{\Lambda}{\min\{\alpha + \mu,\,\theta + \mu\}},\\
&\liminf_{t \to \infty}H_b(t) \ge \dfrac{{\min\{\alpha,\,\theta\}}}{\mu}\dfrac{\Lambda}{\max\{\alpha + \mu,\,\theta + \mu\}}.
\end{split}
\end{equation}
\end{itemize}
\end{theorem}
\begin{proof}
\textbf{Proof of Part (i).} It follows from the system \eqref{eq:1G} that
\begin{equation*}
\begin{split}
 \dfrac{dC}{dt}\bigg|_{C = 0} &= \Lambda + \gamma H_n > 0,\\
 \dfrac{dH_n}{dt}\bigg|_{H_n = 0} &= 0,\\
 \dfrac{d H_b}{dt}\bigg|_{H_b = 0} &= \alpha C + \theta H_n \ge0.
\end{split}
\end{equation*}
Thus, we deduce from \cite[Proposition B.7]{Smith} that $C(t), H_n(t), H_b(t) \ge0$ for $t>0$, whenever $C(0), H_n(0), H_b(0) \ge0$. 
This completes this proof.\\
\textbf{Proof of Part (ii).} 
Let $\big(C(0),\,H_n(0),\,H_b(0)\big)$ be any initial data with $C(0) \ge0, H_n(0) > 0$, and $H_b(0) \ge0$.
We define the function
\begin{equation*}\label{eq:7}
\widetilde{f}(C, H_n) =
\begin{cases}
  &\dfrac{f(C, H_n)}{H_n} \quad  \text{if} \quad H_n > 0,\\[10pt]
  &\dfrac{\partial f(C, 0)}{\partial H_n} \quad \text{if} \quad H_n = 0.
\end{cases}
\end{equation*}
Then, $f(C, H_n) = H_n \widetilde{f}(C, H_n)$ for $H_n\ge0$ and it follows from \textbf{(P1)} that $\widetilde{f}$ is continuous. 
From the second equation of \eqref{eq:1G}, we obtain
\begin{equation*}
   \dfrac{dH_n}{H_n} = \widetilde{f}(C, H_n) - (\gamma + \theta + \mu),
\end{equation*}
which implies that
\begin{equation*}
    H_n(t) = H_n(0)\,
    \exp\Bigl\{\int_0^t\bigl[\widetilde{f}(C(\tau), H_n(\tau)) - (\gamma + \theta + \mu)\bigr]d\tau\Bigr\}.
\end{equation*}
Thus, $H_n(t)>0$ whenever $H_n(0)>0$ and $H_n(t)\equiv0$ if and only if $H_n(0)=0$.

To prove $C(t) > 0$ for $t > 0$, we assume to the contrary that $C(t)=0$ for some $t>0$ and set 
\begin{equation*}
     t^* = \min\{t|\, C(t) = 0\}.
\end{equation*}
At $t=t^*$, we have
\begin{equation*}
     \dfrac{dC}{dt}\bigg|_{t = t^*} = \Lambda + \gamma H_n(t^*) > 0.
\end{equation*}
From the continuity of $C(t)$, there exists $\epsilon > 0$ such that
\begin{equation*}
     C'(t) > 0, \quad t \in (t^* - \epsilon,\,t^* + \epsilon) \subset (0,\,t^*+\epsilon).
\end{equation*}
Thus, for $t^* - \epsilon < t < t^*$, we have
\begin{equation*}
    C(t) < C(t^*) = 0.
\end{equation*}
This is a contradiction to the conclusion in Part (i).
Hence, we conclude that $C(t)>0$ for $t>0$. 
By similar arguments, we also obtain $H_b(0)>0$. 
The proof is complete.\\
\textbf{Proof of Part (iii).} It is easy to verify that the total population $N$ satisfies
\begin{equation*}
   \dfrac{dN}{dt} = \Lambda - \mu N.
\end{equation*}
Thus,
\begin{equation}\label{eq:8}
    N(t) = \bigg(N(0) - \dfrac{\Lambda}{\mu}\bigg)\,\mathrm{e}^{-\mu t} + \dfrac{\Lambda}{\mu},
\end{equation}
which implies the monotone convergence of $N(t)$.\\
\textbf{Proof of Part (iv).} Adding side-by-side the first and second equations of \eqref{eq:1G} gives
\begin{equation*}
\begin{split}
   \Lambda - \max\{\alpha + \mu,\,\theta + \mu\}(C + H_n) \le \dfrac{d(C + H_n)}{dt} 
   &= \Lambda - (\alpha + \mu)C - (\theta + \mu)H_n\\
   &\le\Lambda - \min\{\alpha + \mu,\,\theta + \mu\}(C + H_n).
\end{split}
\end{equation*}
Using the basic comparison theorem for ODEs \cite{McNabb} leads to the first and second estimates of \eqref{eq:9}.

Next, it follows from the third equation of \eqref{eq:1G} that 
\begin{equation*}
\min\{\alpha,\,\theta\}(C + H_n) - \mu H_b \le\dfrac{d H_b}{dt}
= \alpha C + \theta H_n - \mu H_b 
\le\max\{\alpha,\,\theta\}(C + H_n) - \mu H_b.
\end{equation*}
From the first and second estimates of \eqref{eq:9}, we have for $t$ large enough
\begin{equation*}
   \min\{\alpha,\,\theta\}\dfrac{\Lambda}{\max\{\alpha + \mu,\,\theta + \mu\}} - \mu H_b 
   \le \dfrac{dH_b}{dt} 
   \le  \max\{\alpha,\,\theta\}\dfrac{\Lambda}{\min\{\alpha + \mu,\,\theta + \mu\}} - \mu H_b.
\end{equation*}
Then, the two last estimates of \eqref{eq:9} are derived from the basic comparison theorem for ODEs \cite{McNabb}. The proof is complete.
\end{proof}

We now determine the equilibrium points of \eqref{eq:1G}. 
\begin{lemma}\label{Lemma1}
There are two types of equilibrium points of \eqref{eq:1G}:
\begin{itemize} 
\item[(iii)] A hacker-present equilibrium (HPE) point $E^* = \big(C^*,\,H_n^*,\,H_b^*\big)$ exists if and only if 
\begin{equation}\label{eq:14}
     \dfrac{\partial f\bigg(\dfrac{\Lambda}{\alpha + \mu},\,0\bigg)}{\partial H_n} - (\gamma + \theta + \mu) > 0.
\end{equation}
When this is the case, $E^*$ is defined by
\begin{equation}\label{eq:11}
\begin{split}
  H_b^* &= \dfrac{\alpha C^* + \theta H_n^*}{\mu},\\
  C^* &= \dfrac{\Lambda - (\theta + \mu)H_n^*}{\alpha + \mu}, 
\end{split}
\end{equation}
where $H_n^*$ is a unique positive solution, belonging to $\big(0,\,\frac{\Lambda}{\theta + \mu}\big)$, of the equation
\begin{equation}\label{eq:12}
    W(H_n) := \dfrac{1}{H_n}f\bigg(\dfrac{\Lambda - (\theta + \mu)H_n}{\alpha + \mu}, H_n\bigg) - (\gamma + \theta + \mu) = 0.
\end{equation}
\item[(ii)] A hacker-free equilibrium (HFE) point $E^0$ always exists, where
\begin{equation}\label{eq:15}
E^0 = \big(C^0,\,H_n^0,\,H_b^0\big) = \bigg(\dfrac{\Lambda}{\alpha + \mu},\,0,\,\dfrac{\alpha\Lambda}{\mu(\alpha + \mu)}\bigg).
\end{equation}
\end{itemize}
\end{lemma}
\begin{proof}
Any equilibrium point is a solution to the system
\begin{equation}\label{eq:10}
\begin{split}
0 &= \Lambda - f(C, H_n) -\alpha C + \gamma H_n -\mu C, \\
0 &= f(C, H_n) - \gamma H_n - \theta H_n - \mu H_n,\\
0 &= \alpha C + \theta H_n - \mu H_b.
\end{split}
\end{equation}
From the third equation of \eqref{eq:10}, we obtain the first formula of \eqref{eq:11}. 
On the other hand, adding side-by-side the first and second equations of \eqref{eq:10} gives
\begin{equation*}
    \Lambda - (\alpha + \mu)C - (\theta + \mu)H_n = 0,
\end{equation*}
which leads to the second formula of \eqref{eq:11}.

By substituting the formula of $C^*$ in the second formula of \eqref{eq:11} into the the second equation of \eqref{eq:10}, we obtain an equation for $H_n$
\begin{equation*}
    f\bigg(\dfrac{\Lambda - (\theta + \mu)H_n}{\alpha + \mu}, H_n\bigg) - \gamma H_n - \theta H_n - \mu H_n,
\end{equation*}
which is equivalent to $H_nW(H_n) = 0$, where $W$ is defined as in \eqref{eq:12}

Thus, any HPE point must satisfy \eqref{eq:12} and then, $H_n^* \in \bigl(0,\,\frac{\Lambda}{\theta + \mu}\bigr)$. 
Therefore, we consider the function $W$ on $\big(0,\,\frac{\Lambda}{\theta + \mu}\big)$.
On this interval, we have
\begin{equation}\label{eq:13}
\begin{split}
W\bigg(\dfrac{\Lambda}{\theta + \mu}\bigg) &= f\bigg(0,\, \dfrac{\Lambda}{\theta + \mu}\bigg) - (\gamma + \theta + \mu) = -(\gamma + \theta + \mu) < 0,\\
\lim_{H_n \to 0}\bigl(W(H_n)\bigr) &= \dfrac{\partial f\bigg(\dfrac{\Lambda}{\alpha + \mu},\,0\bigg)}{\partial H_n} - (\gamma + \theta + \mu) > 0.
\end{split}
\end{equation}
Note that the first estimate of \eqref{eq:13} is derived from the property $\textbf{(P1)}$, whereas the second one is derived from \eqref{eq:14}.

Furthermore, the first derivative of $W(H_n)$ with $H_n>0$ is given by
\begin{equation*}
\begin{split}
   W'(H_n) &= \dfrac{1}{H_n^2}\bigg[H_n\bigg(-\dfrac{\theta + \mu}{\alpha + \mu}\bigg)\dfrac{\partial f}{\partial C}\bigg(\dfrac{\Lambda - (\theta + \mu)H_n}{\alpha + \mu},\,H_n\bigg)\\
    &\qquad+ H_n\dfrac{\partial f}{\partial H_n}\bigg(\dfrac{\Lambda - (\theta + \mu)H_n}{\alpha + \mu},\,H_n\bigg) - f\bigg(\dfrac{\Lambda - (\theta + \mu)H_n}{\alpha + \mu},\,H_n\bigg)\bigg].
\end{split}
\end{equation*}
From $\textbf{(P2)}$ and $\textbf{(P4)}$, we deduce that $W'(H_n) < 0$ for $H_n \in \big(0,\,\frac{\Lambda}{\theta + \mu}\big)$. 
Combining this with \eqref{eq:13}, we conclude that the equation $W(H_n) = 0$  has a unique solution belonging to $\big(0,\,\frac{\Lambda}{\theta + \mu}\big)$. 
Therefore, $E^*$ is the only HPE point of \eqref{eq:1G}.

From the equation $H_nW(H_n) = 0$, we obtain the trivial solution $H_n^0 = 0$. 
Using \eqref{eq:12} yields $H_b^0$ and $C^0$, which are defined as in  \eqref{eq:15}. 
Thus, $E^0$ is an HFE point, which always exists for all parameter values. 
The proof is complete.
\end{proof}

Before concluding this subsection, we apply the method developed by van den Driessche and Watmough \cite{vdDriessche} to compute the basic reproduction number. 
To do so, we set $X = (H_n,\,H_b,\,C)$ and represent \eqref{eq:1G} in the vector form
\begin{equation*}
    \dfrac{dX}{dt} = \mathcal{F}(X) - \mathcal{V}(x),
\end{equation*}
where
\begin{equation*}
\mathcal{F}(X) =
\begin{bmatrix}
f(C, H_n)\\[2pt]0\\[2pt]0
\end{bmatrix}, \quad
\mathcal{V}(X) =
\begin{bmatrix}
(\gamma + \theta + \mu)H_n\\[2pt]
-\alpha C - \theta H_n + \mu H_b\\[2pt]
-\Lambda + f(C, H_n) + \alpha C - \gamma H_n + \mu C
\end{bmatrix}.
\end{equation*}
At the HFE point $X^0 = (H_n^0,\,H_b^0,\,C^0)$, we have
\begin{equation*}
\begin{split}
D\mathcal{F}(X^0) &= 
\begin{bmatrix}
\dfrac{\partial f(C^0, H_n^0)}{\partial H_n}&0&\dfrac{\partial f(C^0, H_n^0)}{\partial C}\\[4pt]
0&0&0\\[4pt]
0&0&0
\end{bmatrix},\\
D\mathcal{V}(X^0) &=
\begin{bmatrix}
\gamma + \theta + \mu&0&0\\[4pt]
-\theta&\mu&-\alpha\\[4pt]
\dfrac{\partial f(C^0, H_n^0)}{\partial H_n} - \gamma&0&\dfrac{\partial f(C^0, H_n^0)}{\partial C} + \alpha + \mu
\end{bmatrix}.
\end{split}
\end{equation*}
Hence, the basic reproduction number $\mathcal{R}_0$ is computed as
\begin{equation*}
      \mathcal{R}_0 = \dfrac{1}{\gamma + \theta + \mu}\dfrac{\partial f\bigg(\dfrac{\Lambda}{\alpha + \mu},\,0\bigg)}{\partial H_n}.
\end{equation*}
From \eqref{eq:14} in Lemma~\ref{Lemma1}, we see that the HPE point exists if and only if $\mathcal{R}_0>1$. 
This characteristic is commonly observed in compartmental epidemic models \cite{Allen, Brauer, Martcheva}.
As will be seen in the subsequent sections, the asymptotic stability of the equilibrium points also depends on the value of $\mathcal{R}_0$ relative to 1.

\subsection{Local Asymptotic Stability}
In this subsection, the \textit{local asymptotic stability} (LAS) of the equilibrium points is analyzed.
\begin{theorem}[LAS analysis of the HFE point]\label{Lemma2}
The HFE point is locally asymptotically stable if $\mathcal{R}_0 < 1$, and is unstable of $\mathcal{R}_0>1$.
\end{theorem}
\begin{proof}
The Jacobian matrix of \eqref{eq:1G} evaluated at $E^0$ is given by
\begin{equation*}
   J(E^0) = \begin{bmatrix}
   -\dfrac{\partial f(E^0)}{\partial C} - (\alpha + \mu)&-\dfrac{\partial f(E^0)}{\partial H_n} + \gamma&0\\[4pt]
   0& \dfrac{\partial f(E^0)}{\partial H_n} - (\gamma + \theta + \mu)&0\\[4pt]
    \alpha&\theta&-\mu
    \end{bmatrix}.
\end{equation*}
Note that we have used $\frac{\partial f(E^0)}{\partial C} = 0$, which is a direct consequence of the property \textbf{(P1)}. 
Consequently, $J(E^0)$ has three eigenvalues 
\begin{equation*}
\lambda_1 = -\dfrac{\partial f(E^0)}{\partial C} - (\alpha + \mu), \quad \lambda_2 = \dfrac{\partial f(E^0)}{\partial H_n} - (\gamma + \theta + \mu),\quad \lambda_3 = -\mu.
\end{equation*}
It is clear that $\lambda_1$ and $\lambda_3$ are strictly negative. On the other hand, 
\begin{equation*}
\lambda_2 = (\gamma + \theta + \mu)(\mathcal{R}_0 - 1).
\end{equation*}
Hence, $\mathcal{R}_0 < 1$ implies that $\lambda_2 < 0$. 
Consequently, $J(E^0)$ contains three strictly negative real numbers. 
According to Lyapunov's indirect theorem (linearized method) \cite{Khalil, Stuart}, $E^0$ is locally asymptotically stable.

If $\mathcal{R}_0 > 1$, then $\lambda_2 > 0$. This implies that $E^0$ is unstable. 
The proof is complete.
\end{proof}

\begin{theorem}[LAS analysis of the HPE point]\label{Theorem2}
The HPE point $E^*$ of \eqref{eq:1G} is locally asymptotically stable whenever it exists. 
\end{theorem}
\begin{proof}
The Jacobian matrix of \eqref{eq:1G} evaluated at $E^0$ reads
\begin{equation*}
   J(E^*) = \begin{bmatrix}
   -\dfrac{\partial f(E^*)}{\partial C} - (\alpha + \mu)&-\dfrac{\partial f(E^*)}{\partial H_n} + \gamma&0\\[4pt]
   \dfrac{\partial f(E^*)}{\partial C}& \dfrac{\partial f(E^*)}{\partial H_n} - (\gamma + \theta + \mu)&0\\[4pt]
   \alpha&\theta&-\mu
    \end{bmatrix}.
\end{equation*}
Therefore, $J(E^*)$ always has a strictly negative eigenvalue $\lambda_3 = -\mu < 0$, and the two remaining eigenvalues are exactly those associated with the submatrix
\begin{equation*}
J^* =\begin{bmatrix}
-\dfrac{\partial f(E^*)}{\partial C} - (\alpha + \mu)&-\dfrac{\partial f(E^*)}{\partial H_n} + \gamma\\[4pt]
\dfrac{\partial f(E^*)}{\partial C}& \dfrac{\partial f(E^*)}{\partial H_n} - (\gamma + \theta + \mu)\\
\end{bmatrix}.
\end{equation*}

The characteristic polynomial of $J^*$ is
\begin{equation*}
    P_{J^*}(\lambda) = \lambda^2 - \trace(J^*)\lambda + \det(J^*).
\end{equation*}
We will show that
\begin{equation}\label{eq:16}
   \det(J^*) > 0, \quad \trace(J^*) < 0.
\end{equation}
Indeed, since $E^*$ is the positive equilibrium point, 
it satisfies \eqref{eq:10}. Consequently,
\begin{equation*}
   \dfrac{f(C^*, H_n^*)}{H_n^*} = \gamma + \theta + \mu.
\end{equation*}
Using the property \textbf{(P4)} gives
\begin{equation}\label{eq:17}
   \dfrac{\partial f(C^*,\,H_n^*)}{\partial H_n} \le\dfrac{f(C^*,\,H_n^*)}{H_n^*} = \gamma + \theta + \mu,
\end{equation}
which immediately follows that $\trace(J^*) < 0$.

On the other hand, the determinant of $J^*$ is given by
\begin{equation*}
\begin{split}
   \det(J^*) &=  \bigg[-\dfrac{\partial f(E^*)}{\partial C} - (\alpha + \mu)\bigg]\bigg[\dfrac{\partial f(E^*)}{\partial H_n} - (\gamma + \theta + \mu)\bigg] - \bigg[-\dfrac{\partial f(E^*)}{\partial H_n} + \gamma\bigg]\dfrac{\partial f(E^*)}{\partial C}\\
   &= (\mu + \theta)\dfrac{\partial f(E^*)}{\partial C} + (\alpha + \mu)\bigg(\gamma + \theta + \mu - \dfrac{\partial f(E^*)}{\partial H_n}\bigg).
\end{split}
\end{equation*}
Thus, it follows from \textbf{(P2)}  and \eqref{eq:17} that $\det(J^*) > 0$.

Therefore, \eqref{eq:16} has been shown. Using the Routh-Hurwitz criterion (see \cite[Theorem 4.4]{Allen}), we conclude that all the eigenvalues of $J^*$ have a negative real part. 
Consequently, using the linearized method \cite{Khalil, Stuart}, $E^*$ is locally asymptotically stable.
The proof is complete.
\end{proof}

\subsection{Global Asymptotic Stability Analysis of the HFE Point}\label{Subsec3.3}
In this subsection, we establish the GAS of the HFE point.
To do so, we first consider a submodel, which is governed by the first two equations of \eqref{eq:1G} and is given by
\begin{equation}\label{eq:2G}
\begin{split}
 \dfrac{dC}{dt} &= \Lambda - f(C, H_n) -\alpha C + \gamma H_n -\mu C := F_1(C,\,H_n), \\
 \dfrac{dH_n}{dt} &= f(C, H_n) - \gamma H_n - \theta H_n - \mu H_n := F_2(C,\,H_n).
\end{split}
\end{equation}
The HFE and HPE points are then transformed to $\widehat{E}^0 = (C^0,\,\,H_n^0)$ and $\widehat{E}^* = (C^*,\,\,H_n^*)$, respectively. 
Note that Lemma~\ref{Lemma1} implies that all the solutions to \eqref{eq:1G} and \eqref{eq:2G} are bounded.
The following theorem is proved based on the approach developed in \cite{Castillo-Chavez}.
\begin{theorem}[GAS analysis of the HFE point]\label{Theorem3new}
Assume that $\mathcal{R}_0 < 1$. 
Then, the HFE point $E^0$ of \eqref{eq:1G} is not only locally asymptotically stable, but also globally asymptotically stable, whenever
\begin{equation}\label{eq:17new}
\alpha \le\theta.
\end{equation}
\end{theorem}
\begin{proof}
The proof of this theorem has two steps.\\
\textbf{Step 1.} $\widehat{E}^0 = (C^0,\,\,H_n^0)$ is a globally asymptotically stable equilibrium point of \eqref{eq:2G}.\\
Since \eqref{eq:17new} holds, we deduce from Theorem \ref{Theorem1} that
\begin{equation*}
\begin{split}
  \limsup_{t \to \infty}\big(C(t) + H_n(t)\big) &\le \dfrac{\Lambda}{\alpha + \mu} = C^0,\\
  \liminf_{t \to \infty}\big(C(t) + H_n(t)\big) &\ge \dfrac{\Lambda}{\theta + \mu}.
\end{split}
\end{equation*}
Thus, we can consider \eqref{eq:2G} in a feasible set defined by
\begin{equation}\label{eq:Pset}
\Omega = \bigg\{(C,\,H_n) \in \mathbb{R}_+^2\bigg|\dfrac{\Lambda}{\theta + \mu} \le C + H_n \le \dfrac{\Lambda}{\alpha + \mu}\bigg\}.
\end{equation}
Following the approach in \cite{Castillo-Chavez}, we set $x = C$ and $\textbf{I} = H_n$ and represent \eqref{eq:2G} in the form
\begin{equation}\label{eq:2G1}
\begin{split}
& \dfrac{dx}{dt} = \Lambda - f(x, \textbf{I}) -\alpha x + \gamma \textbf{I} -\mu x := F(x,\,\textbf{I}), \\
& \dfrac{d\textbf{I}}{dt} = f(x, \textbf{I}) - \gamma \textbf{I} - \theta \textbf{I} - \mu \textbf{I} := G(x,\,\textbf{I}).
\end{split}
\end{equation}
Then, the HFE point transforms to $U_0 = (x^*,\,0) = \big(C^0,\,0\big)$.
We will verify that the conditions (H1) and (H2) in \cite{Castillo-Chavez} hold for the system \eqref{eq:2G1}.
 Indeed, consider the equation
\begin{equation*}
\dfrac{dx}{dt} = F(x, 0) = \Lambda - \alpha x - \mu x.
\end{equation*}
The exact solution to this system is given by
\begin{equation*}
x(t) = \bigg(x(0) - \dfrac{\Lambda}{\mu + \alpha}\bigg)\mathrm{e}^{-\mu t} + \dfrac{\Lambda}{\alpha + \mu}.
\end{equation*}
This implies that $x^*$ is globally asymptotically stable. Thus, the condition (H1) in \cite{Castillo-Chavez} holds.

On the other hand, if $\mathcal{R}_0 < 1$ then
\begin{equation*}
A := D_{\textbf{I}}G(x^*,\,0) = \dfrac{\partial f(x^*,\,0)}{\partial \textbf{I}} - (\gamma + \theta + \mu) < 0.
\end{equation*}
Then, the second equation of \eqref{eq:2G} can be rewritten as follows:
\begin{equation*}
\dfrac{d\textbf{I}}{dt} = G(x, \textbf{I}) = A\textbf{I} - \widehat{G}(x, \textbf{I}),
\end{equation*}
where 
\begin{equation*}
\widehat{G}(x, \textbf{I}) = \textbf{I}\dfrac{\partial f(x^*,\,0)}{\partial \textbf{I}} - f(x,\textbf{I}).
\end{equation*}
We will show that $\widehat{G}(x, \textbf{I}) \ge 0$ for all $(x,\,\textbf{I}) \in \Omega$. Indeed, it is clear that $\widehat{G}(x, \textbf{I}) = 0$ if $\textbf{I} = 0$.\\
Consider the case $\textbf{I} \ne 0$. It follows from \textbf{(P2)} that 
\begin{equation*}
\widehat{G}(x, \textbf{I}) \ge  \textbf{I}\dfrac{\partial f(x^*,\,0)}{\partial \textbf{I}} - f(C^0,\textbf{I}) =  \textbf{I}\bigg[\dfrac{\partial f(x^*,\,0)}{\partial \textbf{I}} - \dfrac{f(C^0,\textbf{I})}{\textbf{I}}\bigg].
\end{equation*}
Furthermore, we deduce from \textbf{(P4)} that the function  $\frac{f(C^0,\textbf{I})}{\textbf{I}}$ is decreasing for $\textbf{I} > 0$. Consequently,
\begin{equation*}
\dfrac{f(C^0,\textbf{I})}{\textbf{I}} \le \lim_{\textbf{I} \to 0}\dfrac{f(C^0,\textbf{I})}{\textbf{I}} = \dfrac{\partial f(C^0,\,0)}{\partial \textbf{I}}.
\end{equation*}
Thus, we conclude that
\begin{equation*}
\widehat{G}(x, \textbf{I}) \ge \textbf{I}\bigg[\dfrac{\partial f(C^0,\,0)}{\partial \textbf{I}} - \dfrac{f(C^0,\,0)}{\textbf{I}}\bigg]  \ge 0.
\end{equation*}
Therefore, the conditions (H1) and (H2) in \cite{Castillo-Chavez} hold. 
Using the results developed in \cite{Castillo-Chavez}, the GAS of $(x^*,\,0)$ is obtained. \\
\textbf{Step 2:} $E^0$ is a globally asymptotically stable equilibrium point of \eqref{eq:1G}.\\
Since \eqref{eq:1G} has the structure of a cascade system and $\widehat{E}^0$ is a globally asymptotically stable equilibrium point of \eqref{eq:2G}, we only need to consider the following equation
\begin{equation}\label{eq:21}
\dfrac{d H_b}{dt} = \alpha \dfrac{\Lambda}{\alpha + \mu} - \mu H_b,
\end{equation}
which is obtained by substituting $C = C^0$ and $H = H^0$ into the last equation of \eqref{eq:1G}. 
Equation \eqref{eq:21} has a unique equilibrium point $H_b^{E_0} = \dfrac{\Lambda}{\mu(\alpha + \mu)}$. 
The exact solution of \eqref{eq:21} is given by
\begin{equation*}
H_b(t) = \bigg(H_b(0) - H_b^{E_0}\bigg)\mathrm{e}^{-\mu t} + H_b^{E_0}.
\end{equation*}
We can therefore conclude that $H_b^{E_0}$ is a globally asymptotically stable equilibrium point of \eqref{eq:21}.

Thus, sufficient conditions for the global stabilizability of two cascade-connected nonlinear systems \cite{Seibert} imply that: $E^0$ of \eqref{eq:1G} is globally asymptotically stable. The proof is complete.
\end{proof}

The following Theorem~\ref{Theorem3} establishes the GAS of the HFE point using LaSalle's theorem \cite[Theorem 4.4]{Khalil} with the help of a quadratic Lyapunov function.
\begin{theorem}[GAS Analysis Based on LaSalle's Theorem ]\label{Theorem3}
If $\mathcal{R}_0 < 1$, then the HFE point of \eqref{eq:1G} is globally asymptotically stable under the condition \eqref{eq:17new}.
\end{theorem}
\begin{proof}
As in the proof of Theorem~\ref{Theorem3new}, we first prove that $\widehat{E}^0 = (C^0,\,\,H_n^0)$ is a globally asymptotically stable equilibrium point of \eqref{eq:2G}.
To accomplish this, we consider \eqref{eq:2G} on its feasible set $\Omega$ defined in \eqref{eq:Pset} and a Lyapunov function candidate given by
\begin{equation}\label{eq:L1}
V(C, H_n) = \dfrac{1}{2}H_n^2.
\end{equation}
The time derivative of $V$ along with the solutions of \eqref{eq:2G} satisfies
\begin{equation*}
    \dfrac{dV}{dt} = \dfrac{dV}{dH_n}\dfrac{dH_n}{dt} = H_n\big[f(C, H_n) - (\alpha + \gamma + \mu)H_n\big].
\end{equation*}
Consequently, $\frac{dV}{dH_n} = 0$ whenever $H_n = 0$ and for $H_n > 0$, we have
\begin{equation*}
     \dfrac{dV}{dt} = H_n^2\bigg[\dfrac{f(C, H_n)}{H_n} - (\alpha + \gamma + \mu)\bigg].
\end{equation*}
It follows from \textbf{(P4)} that the function $\frac{f(C, H_n)}{H_n}$ is decreasing for $H_n > 0$. 
Combining this with \textbf{(P2)} leads to the estimate
\begin{equation*}
\begin{split}  
   \dfrac{dV}{dt} &= H_n^2\bigg[\dfrac{f(C, H_n)}{H_n} - (\alpha + \gamma + \mu)\bigg] \le H_n^2\bigg[\dfrac{f(C^0, H_n)}{H_n} - (\alpha + \gamma + \mu)\bigg]\\
   &\le H_n^2\bigg[\lim_{H_n \to 0}\dfrac{f(C^0, H_n)}{H_n} - (\alpha + \gamma + \mu)\bigg] = H_n^2\bigg[\dfrac{\partial f(C^0, 0)}{\partial H_n} - (\alpha + \gamma + \mu)\bigg]\\
   &=  (\alpha + \gamma + \mu)(\mathcal{R}_0 - 1)H_n^2.
\end{split}
\end{equation*}
Assuming that $\mathcal{R}_0 < 1$, we can conclude that $\frac{dV}{dH_n} \le 0$ and $\frac{dV}{dt} = 0$ if and only if $H_n = 0$. 
Therefore, the largest invariant set in $E := \Big\{(C, H_n) \in \Omega|\frac{dV}{dt} = 0\Big\}$ is the set
\begin{equation*}
M := \{(C^0, 0)\} = \{\widehat{E}^0\}.
\end{equation*}
According to LaSalle's theorem \cite[Theorem 4.4]{Khalil}, we conclude that $\lim_{t \to \infty}(C(t),\,H_n(t)) = \widehat{E}^0$. 
This immediately implies the GAS of $\widehat{E}^0$.
Finally, by repeating the arguments in Step~2 of the proof of Theorem~\ref{Theorem3new}, we obtain the desired conclusion. The proof is complete.
\end{proof}

\begin{remark}\label{NoteGAS}
Based on well-known global stability results of several epidemic models \cite{Allen, Brauer, Martcheva}, we can make the following conjecture: 
 The HFE point $E^0$ of \eqref{eq:1G} is globally asymptotically stable provided that $\mathcal{R}_0 < 1$. 
This implies that the condition \eqref{eq:17new} may be imposed only for technical reasons. 
This conjecture is supported by numerical simulations conducted in Section~\ref{Sec4}.
\end{remark}

The condition \eqref{eq:17new} ($\alpha \le \theta$) emphasizes the important role of converting black-hat hackers into white-hat hackers in mitigating malicious cyber activities. 
Although increasing $\alpha$ through education and awareness campaigns is beneficial, the results suggest that enhancing $\theta$ may have a more significant and immediate impact. 
Specifically, each successful conversion reduces the number of malicious actors and increases the number of defenders.
Thus, an increase in $\theta$ has a dual influence on the system dynamics, whereas an increase in $\alpha$ only contributes to the growth of the white-hat population. 
This observation reveals the effectiveness of rehabilitation and conversion strategies in accelerating the reduction and eventual eradication of black-hat hackers from the system.

\subsection{Global Stability Analysis of the HPE Point}
Here, we derive the GAS of the HPE equilibrium point by combining the Poincar\'e-Bendixson theorem with the Bendixson-Dulac criterion for two-dimensional dynamical systems.
\begin{theorem}[GAS Analysis of the HPE point]\label{Theorem4}
If $\mathcal{R}_0 > 1$ and $H_n(0) > 0$, then the HPE point $E^*$ is not only locally asymptotically stable but also globally asymptotically stable.
\end{theorem}
\begin{proof}
As a consequence of Theorem \ref{Theorem1}, we only need to establish the GAS of $E^*$ with respect to the interior of $\mathbb{R}_+^3$, i.e., the set $\interior(\mathbb{R}_+^3) = \{(C,\,H_n,\,H_b)\} = \{(C,\,H_n,\,H_b) \in \mathbb{R}^3|C,\,H_n,\,H_b > 0\}$.

The proof of this theorem consists of two steps, as follows.\\
\textbf{Step 1.} $\widehat{E}^* = (C^*,\,\,H_n^*)$ is a globally asymptotically stable equilibrium point of \eqref{eq:2G} with respect to the set $\interior(\mathbb{R}_+^3)$.\\
Consider the Dulac function candidate $D\colon\interior(\mathbb{R}_+^3) \to \mathbb{R}$ given by
\begin{equation}\label{eq:22}
    D(C,\,H_n) = \dfrac{1}{H_n}.
\end{equation}
Then, the function $D$ satisfies
\begin{equation*}
\begin{split}
    \dfrac{\partial(DF_1)}{\partial C} + \dfrac{\partial(DF_2)}{\partial H_n} &= \dfrac{\partial}{\partial C}\bigg(\dfrac{\Lambda}{H_n} - \dfrac{f(C,H_n)}{H_n} - \alpha \dfrac{C}{H_n} + \gamma - \mu\dfrac{C}{H_n}\bigg) + \dfrac{\partial}{\partial H_n}\bigg(\dfrac{f(C,H_n)}{H_n} - (\gamma + \theta + \mu)\bigg)\\
    &= - \dfrac{1}{H_n}\dfrac{\partial f(C, H_n)}{\partial C} - \dfrac{\alpha + \mu}{H_n} + \dfrac{H_n\dfrac{\partial f(C, H_n)}{\partial H_n} - f(C, H_n)}{H_n^2},
\end{split}
\end{equation*}
where $F_1$ and $F_2$ are the right-hand side of \eqref{eq:2G}. We deduce from \textbf{(P4)} that
\begin{equation*}
     \dfrac{\partial(DF_1)}{\partial C} + \dfrac{\partial(DF_2)}{\partial H_n} < 0 
\end{equation*}
for all $(C, H_n) \in\interior(\mathbb{R}_+^2)$. 
According to the Dulac-Bendixson criterion (\cite[Theorem 3.6]{Martcheva}), we
conclude that \eqref{eq:2G} has no periodic orbits or graphs in the first quadrant.

Using the same arguments from the proof of \cite[Theorem~3.8]{Martcheva}, we can show that the HFE equilibrium point $\widehat{E}^0$ does not belong to the omega limit set of $\big(C(0),\,H_n(0))$. 
As a consequence,
\begin{equation*}
    \lim_{t\to\infty}\big(C(t),\,H_n(t)) = \widehat{E}^*.
\end{equation*}
Combining this with the LAS of the HPE equilibrium point established in Theorem~\ref{Theorem2}, we conclude that the GAS of $\widehat{E}^*$ holds.\\
\textbf{Step 2:}  $E^*$ is a globally asymptotically stable equilibrium point of \eqref{eq:1G}.\\
As a consequence of the conclusion in Step~1, we only need to consider the following equation
\begin{equation}\label{eq:23}
    \dfrac{d H_b}{dt} = \alpha C^* + \theta H_n^* - \mu H_b,
\end{equation}
which is obtained by substituting $C = C^*$ and $H = H^*$ into the first two equations of \eqref{eq:1G}.
It is easy to see that \eqref{eq:23} has a unique equilibrium point $H_b^{E_*} = \dfrac{\alpha C^* + \theta H_n^*}{\mu} = H_b^*$. From the formula of the exact solution of \eqref{eq:23}
\begin{equation*}
    H_b(t) = \bigg(H_b(0) - H_b^{E_*}\bigg)\mathrm{e}^{-\mu t} + H_b^{E_*},
\end{equation*}
we conclude that $H_b^{E_*}$ is globally asymptotically stable.

Using sufficient conditions for the global stabilizability of two cascade-connected nonlinear systems \cite{Seibert}, we conclude that $E^*$ of \eqref{eq:1G} is globally asymptotically stable. The proof is complete.
\end{proof}

\section{Construction of Second-Order Positivity-Preserving NSFD Scheme}\label{Sec3}
In this section, we derive  a second-order positivity-preserving NSFD scheme for \eqref{eq:1G} based on the approach proposed in \cite{HoangMatthias2026}. 
To do so, we consider \eqref{eq:1G} on the finite interval $[0,\,T]$ with  strictly positive initial data $C(0), H_n(0), H_b(0) > 0$ and rewrite it in the form
\begin{equation}\label{eq:1Gnew}
\begin{split}
  \dfrac{dC}{dt} &= U_1(C, H_n, H_b) - CV_1(C, H_n, H_b) := R_1(C, H_n, H_b),\\
  \dfrac{dH_n}{dt} &= U_2(C, H_n, H_b) - H_nV_2(C, H_n, H_b) := R_2(C, H_n, H_b),\\
  \dfrac{dH_b}{dt} &= U_3(C, H_n, H_b) - H_bV_3(C, H_n, H_b) := R_3(C, H_n, H_b),
\end{split}
\end{equation}
where $U_i, V_i\colon\interior\big(\mathbb{R}_+^3\big) \to \interior\big(\mathbb{R}_+\big)$ are given by
\begin{equation}\label{eq:2Gnew}
\begin{split}
U_1(C, H_n, H_b) &=  \Lambda + \gamma H_n\\
V_1(C, H_n, H_b) &= \dfrac{f(C, H_n)}{C} + \alpha + \mu,\\
U_2(C, H_n, H_b) &= f(C, H_n),\\
 V_2(C, H_n, H_b) &= \gamma + \theta  + \mu,\\
U_3(C, H_n, H_b) &=  \alpha C + \theta H_n,\\
V_3(C, H_n, H_b) &= \mu.
\end{split}
\end{equation}

Following the approach in \cite{Cresson, Dimitrov3, Wood0}, we obtain a positivity-preserving NSFD scheme for \eqref{eq:1Gnew}: 
\begin{equation}\label{eq:NSFD1}
\begin{split}
\dfrac{C^{k + 1} - C^k}{\phi(\Delta t)} &= U_1(C^k, H_n^k, H_b^k) - C^{k+1}V_1(C^k, H_n^k, H_b^k),\\
\dfrac{H_n^{k + 1} - H_n^k}{\phi(\Delta t)} &= U_2(C^k, H_n^k, H_b^k) - H_n^{k+1}V_2(C^k, H_n^k, H_b^k),\\
\dfrac{H_b^{k + 1} - H_b^k}{\phi(\Delta t)} &= U_3(C^k, H_n^k, H_b^k) - H_b^{k+1}V_3(C^k, H_n^k, H_b^k),
\end{split}
\end{equation}
where $\Delta t=T/N$ with $N \ge 1$ is the step size.
$\big(C^k,\,H_n^k,\,H_b^k\big)^\top$ with $k \ge 1$ are the intended approximations for $\big(C(t^k),\,H_n(t^k),\,H_b(t^k)\big)$ with $t^k = k\Delta t$, respectively.
$\phi(\Delta t)$ satisfies $\phi(\Delta t) = \Delta t + \mathcal{O}(\Delta t^2)$ as $\Delta t\to0$ and  is called a \textit{denominator function}.
Note that \eqref{eq:NSFD1} is equivalent to
\begin{equation*}
\begin{split}
C^{k + 1} &=  \dfrac{C^k + \phi(\Delta t)U_1(C^k, H_n^k, H_b^k)}{1 + \phi(\Delta t)V_1(C^k, H_n^k, H_b^k)},\\
H_n^{k + 1} &=  \dfrac{H_n^k + \phi(\Delta t)U_2(C^k, H_n^k, H_b^k)}{1 + \phi(\Delta t)V_2(C^k, H_n^k, H_b^k)},\\
H_b^{k + 1} &=  \dfrac{H_b^k + \phi(\Delta t)U_3(C^k, H_n^k, H_b^k)}{1 + \phi(\Delta t)V_3(C^k, H_n^k, H_b^k)},
\end{split}
\end{equation*}
which is explicit and satisfies $C^{k}, H_n^k, H_b^k > 0$ for $k>0$ whenever $C^0, H_n^0, H_b^0 > 0$. Therefore, \eqref{eq:NSFD1} preserves the positivity of the solutions for all $\Delta t>0$.
However, as mentioned in \cite{Cresson}, the NSFD scheme of the form \eqref{eq:NSFD1} is only convergent of order 1. 

Another first-order positivity-preserving scheme can be derived from the generalized NSFD method constructed by Wood and Kojouharov \cite{Wood1}. 
This scheme examines the sign of the right-hand side function of \eqref{eq:1G} at each iteration step and is only first-order convergent (see Table~\ref{Table4}).

Using the approach introduced in \cite{HoangMatthias2026}, we formulate another NSFD scheme for \eqref{eq:1Gnew} in the form
\begin{align*}
\dfrac{C^{k + 1} - C^k}{\Phi_1(\Delta t, C^k, H_n^k, H_b^k)} 
&= U_1(C^k, H_n^k, H_b^k) - C^{k+1}V_1(C^k, H_n^k, H_b^k)\\
&\quad+ \dfrac{\Delta t}{2} A_1(C^k, H_n^k, H_b^k) - \dfrac{\Delta t}{2} \dfrac{B_1\big(C^k, H_n^k, H_b^k\big)}{C^k},\\
\dfrac{H_n^{k + 1} - H_n^k}{\Phi_2\big(\Delta t, C^k, H_n^k, H_b^k\big)} 
&= U_2(C^k, H_n^k, H_b^k) - H_n^{k+1}V_2(C^k, H_n^k, H_b^k)
\stepcounter{equation}\tag{\theequation}\label{eq:NSFD2}\\
&\quad+ \dfrac{\Delta t}{2} A_2(C^k, H_n^k, H_b^k) - \dfrac{\Delta t}{2} \dfrac{B_2(C^k, H_n^k, H_b^k)}{H_n^k},\\
\dfrac{H_b^{k + 1} - H_b^k}{\Phi_3\big(\Delta t, C^k, H_n^k, H_b^k\big)} 
&= U_3(C^k, H_n^k, H_b^k) - H_b^{k+1}V_3(C^k, H_n^k, H_b^k)\\
&\quad+ \dfrac{\Delta t}{2} A_3(C^k, H_n^k, H_b^k) - \dfrac{\Delta t}{2} \dfrac{B_3(C^k, H_n^k, H_b^k)}{H_n^k},
\end{align*}
where $A_i, B_i\colon \interior\big(\mathbb{R}_+^3\big) \to\interior\big(\mathbb{R}_+\big)$ are defined by
\begin{align*}
A_1(C, H_n, H_b) &= \bigg(\dfrac{\partial f(C, H_n)}{\partial C} + \mu + \alpha\bigg)\big[f(C, H_n) + (\mu + \alpha)C\big] + \dfrac{f(C, H_n)}{\partial H_n}(\gamma + \theta + \mu)H_n + \gamma f(C, H_n),\\
B_1(C, H_n, H_b) &= \bigg(\dfrac{\partial f(C, H_n)}{\partial C} + \mu + \alpha\bigg)(\Lambda + \gamma H_n) + \gamma(\gamma + \theta + \mu)H_n + \dfrac{\partial f(C, H_n)}{\partial H_n}f(C, H_n),\\
A_2(C, H_n, H_b) &= \dfrac{\partial f(C, H_n)}{\partial C}(\Lambda + \gamma H_n) + \dfrac{\partial f(C, H_n)}{\partial H_n}f(C, H_n) + (\gamma + \theta  + \mu)^2 H_n,
\stepcounter{equation}\tag{\theequation}\label{eq:NSFD3}\\
B_2(C, H_n, H_b) &= \dfrac{\partial f(C, H_n)}{\partial C}\big[f(C, H_n) + (\alpha + \mu)C\big] +  \dfrac{\partial f(C, H_n)}{\partial H_n}(\gamma + \theta + \mu)H_n + (\gamma + \theta + \mu)f(C, H_n),\\
A_3(C, H_n, H_b) &= \alpha(\Lambda + \gamma H_n) + \theta f(C, H_n) + \mu^2H_b,\\
B_3(C, H_n, H_b) &= \alpha\big[f(C, H_n) - (\alpha + \mu)C\big] + \theta(\gamma + \theta + \mu)H + \mu(\alpha C + \theta H),
\end{align*}
and the denominator functions $\Phi_i(\Delta t, C, H_n)$ are given by 
\begin{equation}\label{eq:NSFD4}
\begin{split}
\Phi_1(\Delta t, C, H_n, H_b) &= \dfrac{\mathrm{e}^{2V_1(C, H_n, H_b)\Delta t} - 1}{2V_1(C, H_n, H_b)},\\
\Phi_2(\Delta t, C, H_n, H_b) &= \dfrac{\mathrm{e}^{2V_2(C, H_n, H_b)\Delta t} - 1}{2V_2(C, H_n, H_b)},\\
\Phi_3(\Delta t, C, H_n, H_b) &= \dfrac{\mathrm{e}^{2V_3(C, H_n, H_b)\Delta t} - 1}{2V_3(C, H_n, H_b)}.
\end{split}
\end{equation}
Note that the functions $A_i, B_i$ and $\Phi_i$ ($i = 1, 2, 3$) satisfy
\begin{equation*}
\begin{split}
&A_i(C, H_n, H_b) + B_i(C, H_n, H_b) \\
&= \dfrac{\partial R_i(C, H_n, H_b)}{\partial C}R_1(C, H_n, H_b) + \dfrac{\partial R_i(C, H_n, H_b)}{\partial H_n}R_2(C, H_n, H_b) + \dfrac{\partial R_i(C, H_n, H_b)}{\partial H_b}R_3(C, H_n, H_b),\\
&\dfrac{\partial^2\Phi_i(\Delta t, C, H_n, H_b)}{\partial \Delta t^2}\bigg|_{\Delta t = 0} = V_i(C, H_n, H_b), 
\end{split}
\end{equation*}
where $R_i(C, H_n, H_b)$ ($i = 1, 2, 3$) are defined in \eqref{eq:1Gnew}--\eqref{eq:2Gnew}, respectively.

Furthermore, \eqref{eq:NSFD1} can be transformed into the explicit form
\begin{align*}
C^{k + 1} &=  \dfrac{C^k + \Phi_1(\Delta t, C^k, H_n^k, H_b^k)U_1(C^k, H_n^k, H_b^k) + \dfrac{\Delta t}{2}\Phi_1(\Delta t, C^k, H_n^k, H_b^k)A_1(C^k, H_n^k, H_b^k)}{1 + \Phi_1(\Delta t, C^k, H_n^k, H_b^k)V_1(C^k, H_n^k, H_b^k) + \dfrac{\Delta t}{2}\Phi_1(\Delta t, C^k, H_n^k, H_b^k)\dfrac{B_1(C^k, H_n^k, H_b^k)}{C^k}},\\
H_n^{k + 1} &=  \dfrac{H_n^k + \Phi_2(\Delta t, C^k, H_n^k, H_b^k)U_2(C^k, H_n^k, H_b^k) + \dfrac{\Delta t}{2}\Phi_2(\Delta t, C^k, H_n^k, H_b^k)A_2(C^k, H_n^k, H_b^k)}{1 + \Phi_2(\Delta t, C^k, H_n^k, H_b^k)V_2(C^k, H_n^k, H_b^k) + \dfrac{\Delta t}{2}\Phi_2(\Delta t, C^k, H_n^k, H_b^k)\dfrac{B_2(C^k, H_n^k, H_b^k)}{H_n^k}},
\stepcounter{equation}\tag{\theequation}\label{eq:NSFD5}\\
H_b^{k + 1} &=  \dfrac{H_b^k + \Phi_3(\Delta t, C^k, H_n^k, H_b^k)U_3(C^k, H_n^k, H_b^k) + \dfrac{\Delta t}{2}\Phi_3(\Delta t, C^k, H_n^k, H_b^k)A_3(C^k, H_n^k, H_b^k)}{1 + \Phi_3(\Delta t, C^k, H_n^k, H_b^k)V_3(C^k, H_n^k, H_b^k) + \dfrac{\Delta t}{2}\Phi_3(\Delta t, C^k, H_n^k, H_b^k)\dfrac{B_3(C^k, H_n^k, H_b^k)}{H_b^k}},
\end{align*}
which is useful in computing and implementing \eqref{eq:NSFD1}.

Based on \cite[Theorems 1 and 2]{HoangMatthias2026}, we obtain the following result.
\begin{theorem}[Second-Order Positivity-Preserving NSFD Scheme]\label{TheoremNSFD}
The following assertions hold for the NSFD scheme \eqref{eq:NSFD1}:
\begin{itemize}
\item[(i)] It preserves the positivity of the solutions for all finite step size, i.e.\ for any $\Delta t > 0$
\begin{equation*}
C^0,\, H_n^0,\, H_b^0 > 0 \Longrightarrow C^k, H_n^k, H_b^k > 0, \quad k > 0.
\end{equation*}
\item[(ii)] It is convergent of order two.
\end{itemize}
\end{theorem}
In  the next section, numerical experiments will be conducted using the second-order positivity-preserving NSFD scheme \eqref{eq:NSFD1} to simulate the dynamical behaviour of the continuous-time model \eqref{eq:1G}.
The results suggest that the NSFD scheme preserves the positivity and asymptotic stability of \eqref{eq:1G}.

\section{Numerical Experiments}\label{Sec4}
In this section, we perform numerical experiments to validate  the theoretical findings. To this end, we consider the Beddington-DeAngelis response function \cite{Beddington, DeAngelis}
\begin{equation}\label{eq:Nonlinear}
   f(C, H_n) = \dfrac{\beta CH_n}{1 + \kappa_1 C + \kappa_2 H_n}, \quad \beta > 0, \quad \kappa_i \ge0.
\end{equation}
The basic reproduction number of \eqref{eq:1G} is computed by
\begin{equation*}
\mathcal{R}_0 = \dfrac{1}{\gamma + \theta + \mu}\dfrac{\beta\Lambda}{\bigg(1 + \kappa_1\dfrac{\Lambda}{\alpha + \mu}\bigg)(\alpha + \mu)}.
\end{equation*}

\subsection{An Error Analysis of the NSFD Schemes}
In this subsection, we provide an error analysis of the first- and second-order NSFD schemes formulated in Section~\ref{Sec3}. 
Finally, we consider \eqref{eq:1G} with the following parameters
\begin{equation*}
\begin{split}
  &\Lambda = 10^4, \quad  \alpha = 0.025, \quad\quad \gamma = 0.02, \quad \mu = 0.01,\\
  &\theta = 0.025, \quad \beta = 10^{-6}, \quad \kappa_1 = 0.01, \quad \kappa_2 = 0.02,
\end{split}
\end{equation*}
and the initial data
\begin{equation*}
     C(0) = 10^4, \quad H_n(0) = 10^3, \quad H_b(0) = 10^3.
\end{equation*}
To estimate errors and the \textit{rate of convergence} (ROC), we use the numerical approximation obtained by employing an 11-stage, eighth-order Runge-Kutta method \cite{Copper} with a step size $\Delta t = 10^{-6}$ on the interval $[0,\, 1]$, as a reference solution. 
Then, we compute the errors as
\begin{equation*}
\begin{split}
&e^k := \big|C^k - C(t^k)\big| + \big|H_n^k - H_n(t^k)\big| + \big|H_b^k - H_b(t^k)\big|, 
\qquad t^k = k\Delta t,\,\,\,\Delta t=\dfrac{1}{N},\,\,\,0 \le k \le N,\\
&\err^{\rm max} = \max_{k}e^k,\\
&\err^{\rm final} = e^N, \\
&\err^{\rm average} = \dfrac{\mathlarger{\sum}_{k = 1}^Ne_k}{N},\\
&\err^{\rm relative} = \dfrac{\big|C^N - C(t^N)\big|}{\big|C(t^N)\big|} + \dfrac{\big|H_n^N - H_n(t^N)\big|}{\big|H_n(t^N)\big|} + \dfrac{\big|H_b^N - H_b(t^N)\big|}{\big|H_b(t^N)\big|},
\end{split}
\end{equation*}
where $\big(C(t),\,H_n(t),\,H_b(t)\big)$ stands for the reference solution. 
Additionally, the ROC is approximated as (see \cite{Ascher}):
\begin{equation*}
ROC := \log_{\bigg(\dfrac{\Delta t_1}{\Delta t_2}\bigg)}\bigg(\dfrac{\err^{\rm final}(\Delta t_1)}{\err^{\rm final}(\Delta t_2)}\bigg).
\end{equation*}

The errors and ROC of the first- and second-order NSFD schemes \eqref{eq:NSFD2} and \eqref{eq:NSFD1}, the Wood-Kojouharov's scheme \cite{Wood1} and the standard explicit Euler scheme \cite{Ascher, Stuart}, are reported in Tables~\ref{Table1}--\ref{Table4}, respectively.

The results demonstrate that the second-order positivity-preserving NSFD scheme \eqref{eq:NSFD2} significantly improves the accuracy of the remaining three first-order schemes. 
Furthermore, it preserves the positivity of the solutions regardless of the step size. 
As will be demonstrated in Subsection \ref{Subsec4.2}, \eqref{eq:NSFD2} preserves not only the positivity but also the asymptotic stability properties of the continuous-time model. 
Therefore, it is particularly useful for simulating the dynamics of continuous-time models over long time periods.
\begin{table}[htb]
\centering
\caption{The errors and ROC of the second-order positivity-preserving NSFD scheme \eqref{eq:NSFD2}.}\label{Table1}
\begin{tabular}{cccccccccccccccccccc}
\hline
$\Delta t$ & $\err^{\rm max}$ & $\err^{\rm average}$ & $\err^{\rm final}$ & $\err^{\rm relative}$ & ROC \\
\hline
$2 \times 10^{-1}$ & 1.2707 & 0.6465 & 1.2707 & 5.7179e-005 &\\
$10^{-1}$& 0.3156 & 0.1609 & 0.3156 & 1.4176e-005 & 2.0096\\
$5 \times 10^{-2} $ & 0.0786 &  0.0401 & 0.0786 & 3.5291e-006 & 2.0048 \\
$10^{-2}$ & 0.0031 &  0.0016 & 0.0031 & 1.4069e-007 & 2.0017\\
$5 \times 10^{-3}$ & 7.8397e-004 & 4.0036e-004 & 7.8397e-004 & 3.5158e-008 & 2.0005 \\
$10^{-3}$ & 3.1358e-005 & 1.6015e-005 & 3.1358e-005 &  1.4066e-009 & 2.0000\\
$5 \times 10^{-4}$ &  7.8647e-006 & 4.0169e-006 & 7.8647e-006 & 3.5372e-010 & 1.9954\\
$10^{-4}$ & 2.9680e-007 & 1.5147e-007 & 2.9680e-007 &  1.2775e-011 & 2.0362\\
\hline
\end{tabular}
\end{table}

\begin{table}[htb]
\centering
\caption{The errors and ROC of the first-order positivity-preserving NSFD scheme \eqref{eq:NSFD1} using $\phi(\Delta t) = \Delta t$.}\label{Table2}
\begin{tabular}{cccccccccccccccccccc}
\hline
$\Delta t$ & $\err^{\rm max}$ & $\err^{\rm average}$ & $\err^{\rm final}$ & $\err^{\rm relative}$ & ROC \\
\hline
$2 \times 10^{-1}$ &  42.2252 & 21.2472 & 42.2252 &  0.0019 &\\
$10^{-1}$ & 21.1528 & 10.6524 &  21.1528 & 9.5738e-004 & 0.9973\\
$5 \times 10^{-2}$ &  10.5865 &  5.3334 & 10.5865 & 4.7909e-004 & 0.9986\\
$10^{-2}$ &  2.1189 &  1.0678 &  2.1189 & 9.5884e-005 &  0.9995\\
$5 \times 10^{-3}$ &  1.0596 & 0.5340 & 1.0596 & 4.7946e-005 &  0.9999\\
$10^{-3}$ & 0.2119 & 0.1068 &  0.2119 & 9.5899e-006 &  1.0000\\
$5 \times 10^{-4}$ &  0.1060 & 0.0534 & 0.1060 & 4.7950e-006 & 1.0000\\
$10^{-4}$ & 0.0212 & 0.0107 & 0.0212 & 9.5900e-007 &  1.0000\\
\hline
\end{tabular}
\end{table}

\begin{table}[htb]
\centering
\caption{The errors and ROC of the Wood-Kojouharov's method \cite{Wood1} using $\phi(\Delta t) = \Delta t$.}\label{Table3}
\begin{tabular}{cccccccccccccccccccc}
\hline
$\Delta t$ & $\err^{\rm max}$ & $\err^{\rm average}$ & $\err^{\rm final}$ & $\err^{\rm relative}$ & ROC \\
\hline
$2 \times 10^{-1}$ &  38.7350 & 19.5973 & 38.7350 &  0.0015 &\\
$10^{-1}$ & 19.3212 & 9.7895 & 19.3212 & 7.3549e-004 &  1.0035 \\
$5 \times 10^{-2}$ &  9.6491 &  4.8925 & 9.6491 & 3.6733e-004 & 1.0017\\
$10^{-2}$ &  1.9280 & 0.9781 &  1.9280 & 7.3400e-005 & 1.0006\\
$5 \times 10^{-3}$ &  0.9639 & 0.4890 & 0.9639 & 3.6696e-005 & 1.0002\\
$10^{-3}$ & 0.1928 & 0.0978 &  0.1928 &  7.3385e-006 & 1.0001 \\
$5 \times 10^{-4}$ &  0.0964 & 0.0489 &  0.0964 & 3.6692e-006 & 1.0000\\
$10^{-4}$ &  0.0193 &   0.0098 & 0.0193 & 7.3383e-007 &  1.0000\\
\hline
\end{tabular}
\end{table}

\begin{table}[htb]
\centering
\caption{The errors and ROC of the standard explicit Euler scheme.}\label{Table4}
\begin{tabular}{cccccccccccccccccccc}
\hline
$\Delta t$ & $\err^{\rm max}$ & $\err^{\rm average}$ & $\err^{\rm final}$ & $\err^{\rm relative}$ & ROC \\
\hline
$2 \times 10^{-1}$ &  38.7912 & 19.6238 & 38.7912 & 0.0015 &\\
$10^{-1}$ & 19.3399 & 9.7978 & 19.3399 & 7.3907e-004 & 1.0042 \\
$5 \times 10^{-2}$ &  9.6561 & 4.8954 &  9.6561 & 3.6891e-004 & 1.0021\\
$10^{-2}$ &  1.9290 & 0.9785 & 1.9290 & 7.3684e-005 &   1.0007 \\
$5 \times 10^{-3}$ &  0.9644 & 0.4892 &  0.9644 & 3.6836e-005 & 1.0002\\
$10^{-3}$ & 0.1929 & 0.0978 & 0.1929 & 7.3661e-006 & 1.0001\\
$5 \times 10^{-4}$ &  0.0964 & 0.0489 &  0.0964 & 3.6830e-006 & 1.0000\\
$10^{-4}$ & 0.0193 & 0.0098 & 0.0193 & 7.3659e-007 & 1.0000\\
\hline
\end{tabular}
\end{table}

\subsection{Dynamics of the Continuous-Time model by the NSFD Scheme}\label{Subsec4.2}
In this subsection, we use the second-order positivity-preserving NSFD scheme \eqref{eq:NSFD2} to simulate the dynamics of the continuous-time model \eqref{eq:1G}.
To do so, we consider \eqref{eq:1G} with the parameters given in Table~\ref{Table5}.
The numerical approximations, generated by implementing \eqref{eq:NSFD2} with three different step sizes on the interval $[0,\,\,\,10^3]$, are depicted in Figures~\ref{Fig:1}--\ref{Fig:3}. 
Each blue curve in these figures represents a solution trajectory in the phase space associated with a specific initial condition. 
The red circle marks the position of the globally asymptotically stable equilibrium (HFE or HPE), whereas the yellow arrows indicate the direction of the system evolution.

The numerical results show that the qualitative behavior of the computed solutions aligns well with the theoretical analysis established in Section \ref{Sec2}. 
Furthermore, \eqref{eq:NSFD2} preserves the positivity of the solutions and the asymptotic stability properties of the continuous-time model \eqref{eq:1G}. 
These findings further demonstrate the advantages of the second-order positivity-preserving NSFD approach.

Note that \eqref{eq:17new} does not hold for the parameter set 2 in Table \ref{Table5}. 
However, the HFE point remains globally asymptotically stable as shown in Figure \ref{Fig:2}. This supports the conjecture given in Remark~\ref{NoteGAS}.

\begin{remark}
The Beddington-DeAngelis response function \eqref{eq:Nonlinear} reduces to the bilinear incidence function and the saturated incidence function if $\kappa_1 = \kappa_2 = 0$ and $\kappa_1 = 0$, respectively.
Therefore, \eqref{eq:1G} can provide a wider range of realistic scenarios. 
In particular, if $f$ is considered as a control parameter, this feature is useful for parameter estimation problems.
\end{remark}

\begin{table}[htb]
\centering
\caption{The sets of the parameters used in numerical computation}\label{Table5}
\begin{tabular}{cccccccccccccccccccccccccc}
\hline
Set& $\Lambda$ & $\alpha$ & $\gamma$ & $\mu$ & $\theta$  & $\beta$ & $\kappa_1$ & $\kappa_2$ & $\mathcal{R}_0$\\[2pt]
\hline
1& $10^4$ & 0.025 & 0.02 & 0.01 & 0.03 & $ 5.5 \times 10^{-5}$ & 0.001 & 0.002 & 0.9135 \\[2pt]
\hline
2 & $10^5$ & 0.02 & 0.025 & 0.02 & 0.01 & $10^{-4}$ & 0.005 & 0.005 & 0.3636\\[2pt]
\hline
3 & $10^6$ & 0.02 & 0.025 & 0.02 & 0.01 & $10^{-3}$ & 0.01 & 0.02 &  1.8182 \\[2pt]
\hline
\end{tabular}\vspace*{0.5cm}
\begin{tabular}{cccccccccccccccccccccccccccc}
\hline
Set&  GAS equilibrium point & Remark\\[2pt]
\hline
1&  $(2.8571 \times 10^5,\,\,\, 0,\,\,\, 7.1429 \times 10^5)$ & \eqref{eq:17new} holds\\[2pt]
\hline
2 & $(25 \times 10^5,\,\,\,0,\,\,\,25 \times 10^5)$&{\eqref{eq:17new} does not hold}\\[2pt]
\hline
3 &  $(1.9130 \times 10^7,\,\,\, 7.8260 \times 10^6,\,\,\,  2.3043 \times 10^7)$& $\mathcal{R}_0 > 1$\\[2pt]
\hline
\end{tabular}
\end{table}

\section{Concluding Remarks and Discussions}
The first conclusion of this work is that we have proposed and analyzed a generalized compartment model representing hacker dynamics in cybersecurity systems. 
This model extends a recognized model formulated in \cite{Hassouni}.
The proposed model uses general nonlinear incidence rate functions instead of bilinear ones, which makes it more flexible and applicable to a wider range of scenarios.  
After proposing the generalized model, we investigated the positivity and boundedness of its solutions. 
We also computed the basic reproduction number and determined the hacker-free and hacker-persistent equilibrium points. 
Finally, we established their local and global asymptotic stability (GAS) properties. 
Consequently, the global dynamics of the generalized compartment model have been fully characterized. 
The proposed framework is not limited to cybersecurity applications, but can also be used to study the spread of information on the internet and online social networks.

Second, we adopted the approach introduced in \cite{HoangMatthias2026} to derive a second-order positivity-preserving (NSFD) scheme for numerical simulations. 
Our analysis and numerical simulations showed that the derived NSFD scheme significantly improves the accuracy of first-order NSFD schemes.
Moreover, the NSFD scheme preserves the positivity and asymptotic stability properties of the continuous-time model. 
Therefore, it is particularly appropriate for simulating the dynamics of continuous-time models over long time periods.

Lastly, several numerical experiments were conducted to validate the theoretical findings and demonstrate the advantages of the implemented NSFD scheme. 
The numerical results strongly support the theoretical assertions.

In the near future, we plan to explore the practical applications of these findings. 
Conversely, constructing higher-order numerical methods that preserve the dynamical properties of the generalized continuous-time model will be a topic of interest.

\begin{figure}[H]
\subfloat[$\Delta t = 5.0$]{%
\includegraphics[height=10cm,width=9cm]{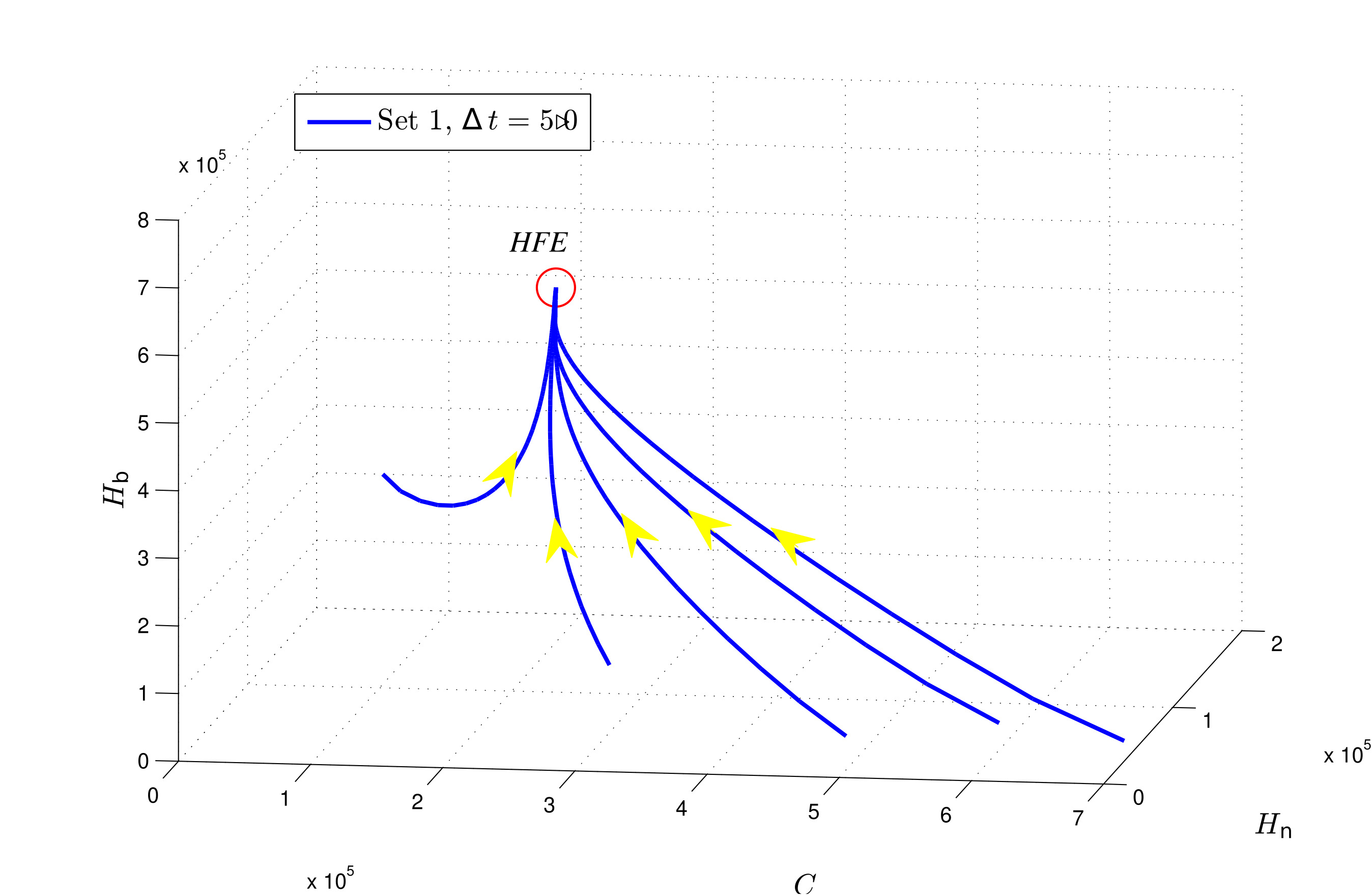}
\label{Figure:1a}
}\hfill
\subfloat[$\Delta t = 1.0$]{%
\includegraphics[height=10cm,width=9cm]{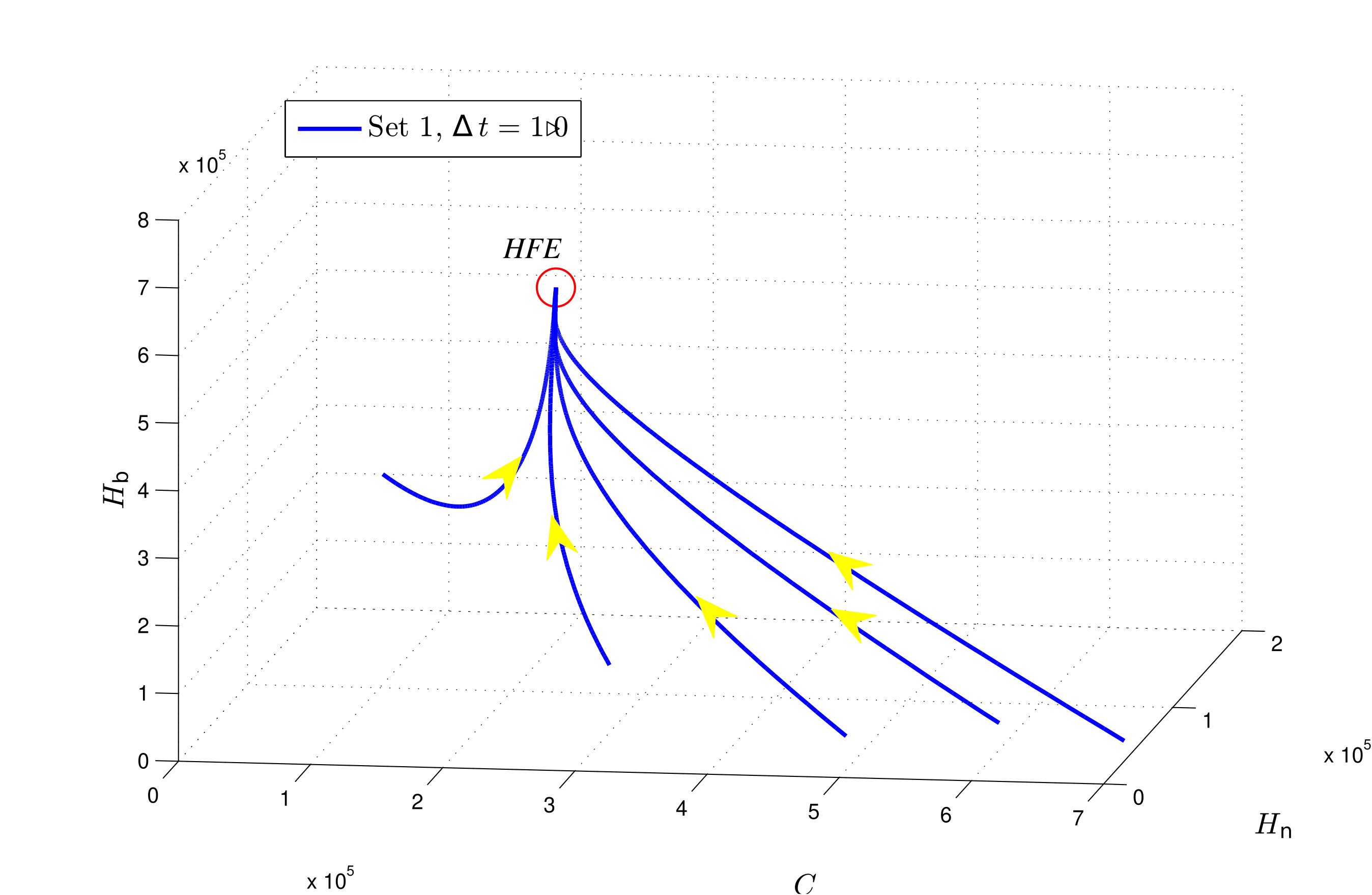}
\label{Figure:1b}
}\hfill
\subfloat[$\Delta t = 10^{-3}$]{%
\includegraphics[height=10cm,width=16cm]{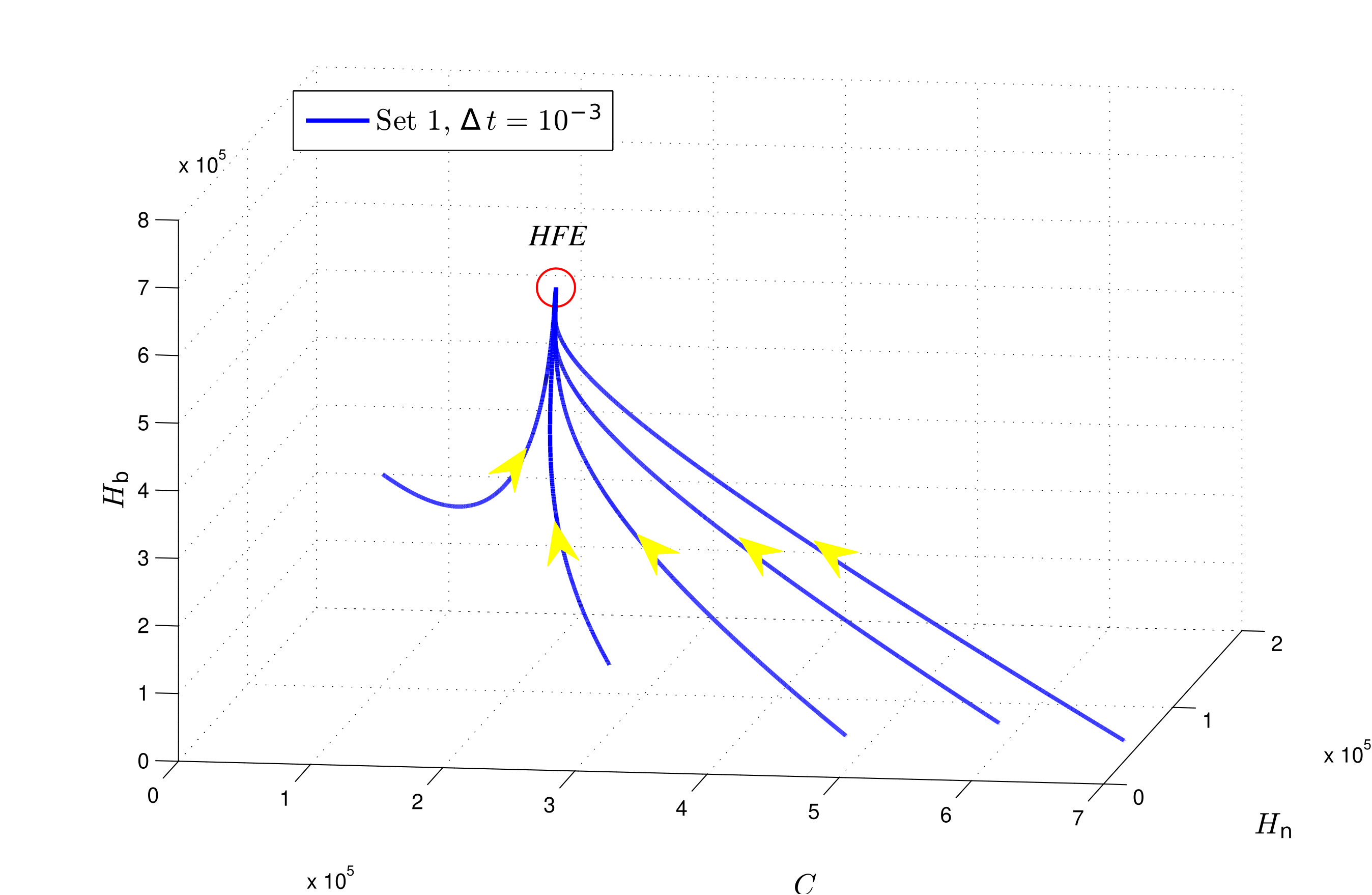}
\label{Figure:1c}
}
\caption{The numerical solutions  generated  by the second-order positivity-preserving NSFD scheme \eqref{eq:NSFD2} using parameter set 1 in Table~\ref{Table5}.}
\label{Fig:1}
\end{figure}

\begin{figure}[H]
\subfloat[$\Delta t = 5.0$]{%
\includegraphics[height=10cm,width=9cm]{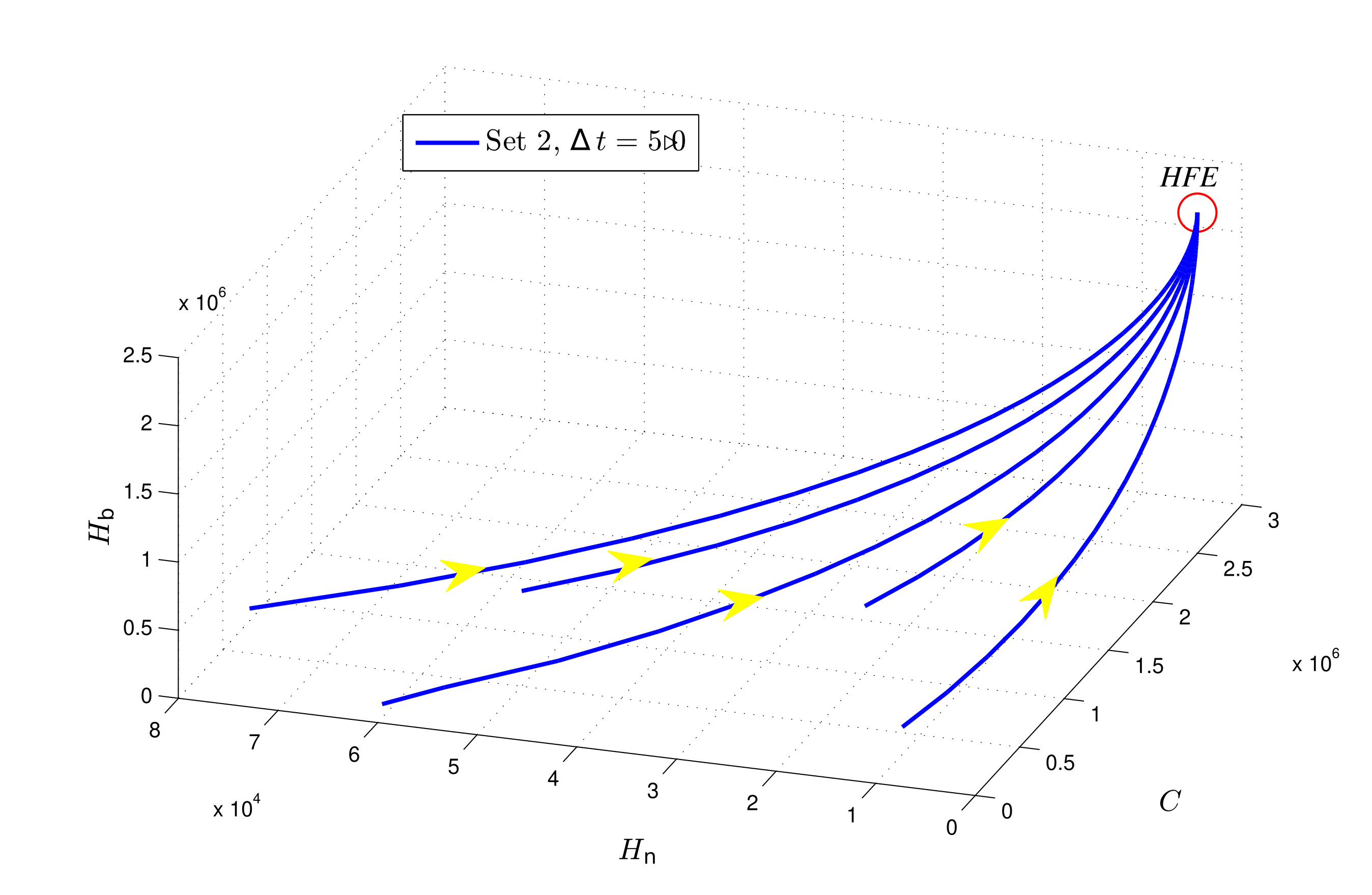}
\label{Figure:2a}
}\hfill
\subfloat[$\Delta t = 1.0$]{%
\includegraphics[height=10cm,width=9cm]{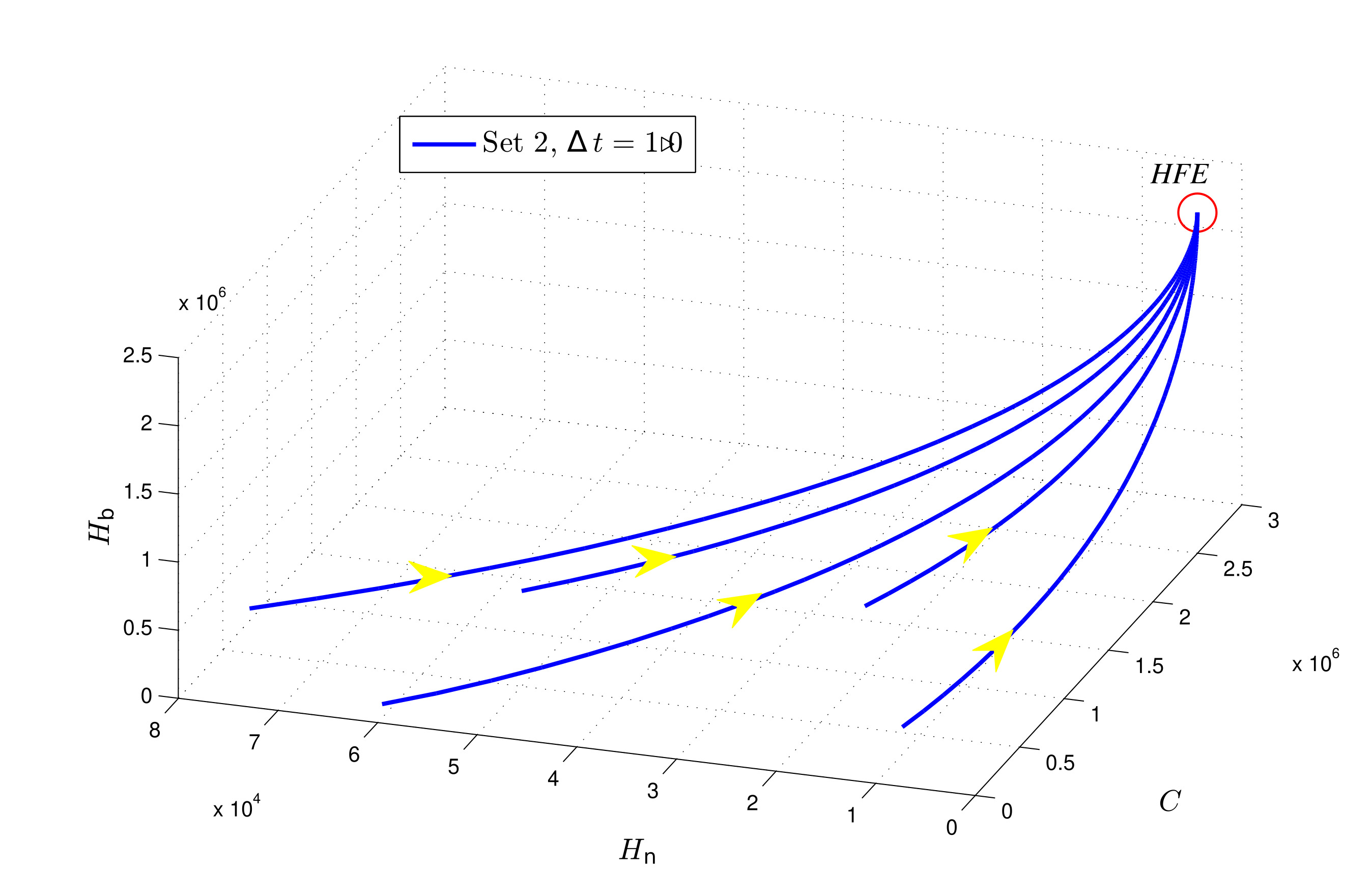}
\label{Figure:2b}
}\hfill
\subfloat[$\Delta t = 10^{-3}$]{%
\includegraphics[height=10cm,width=16cm]{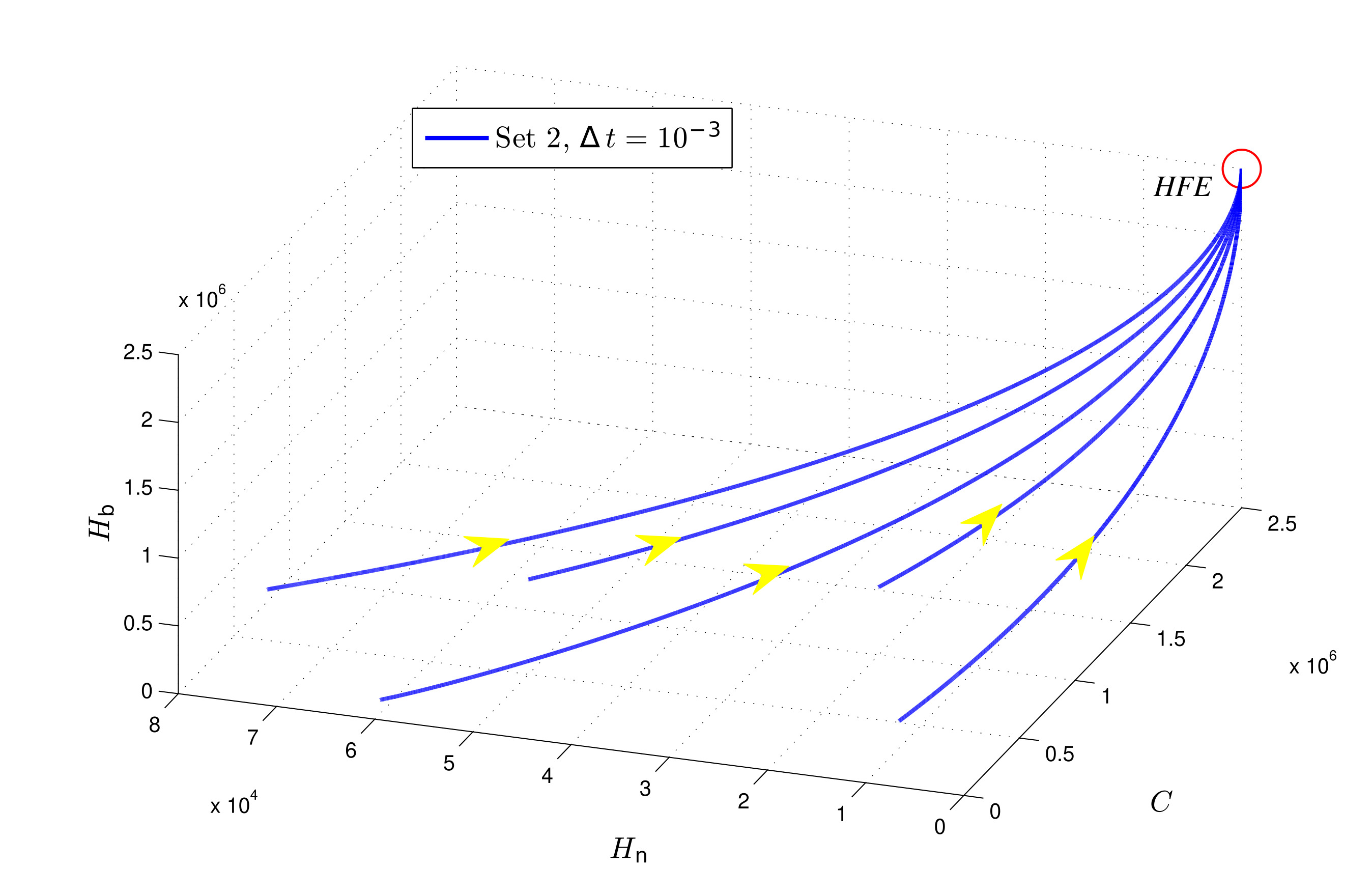}
\label{Figure:2c}
}
\caption{The numerical solutions  generated  by the second-order positivity-preserving NSFD scheme \eqref{eq:NSFD2} using parameter set $2$ in Table \ref{Table5}.}
\label{Fig:2}
\end{figure}
\begin{figure}[H]
\subfloat[$\Delta t = 5.0$]{%
\includegraphics[height=10cm,width=9cm]{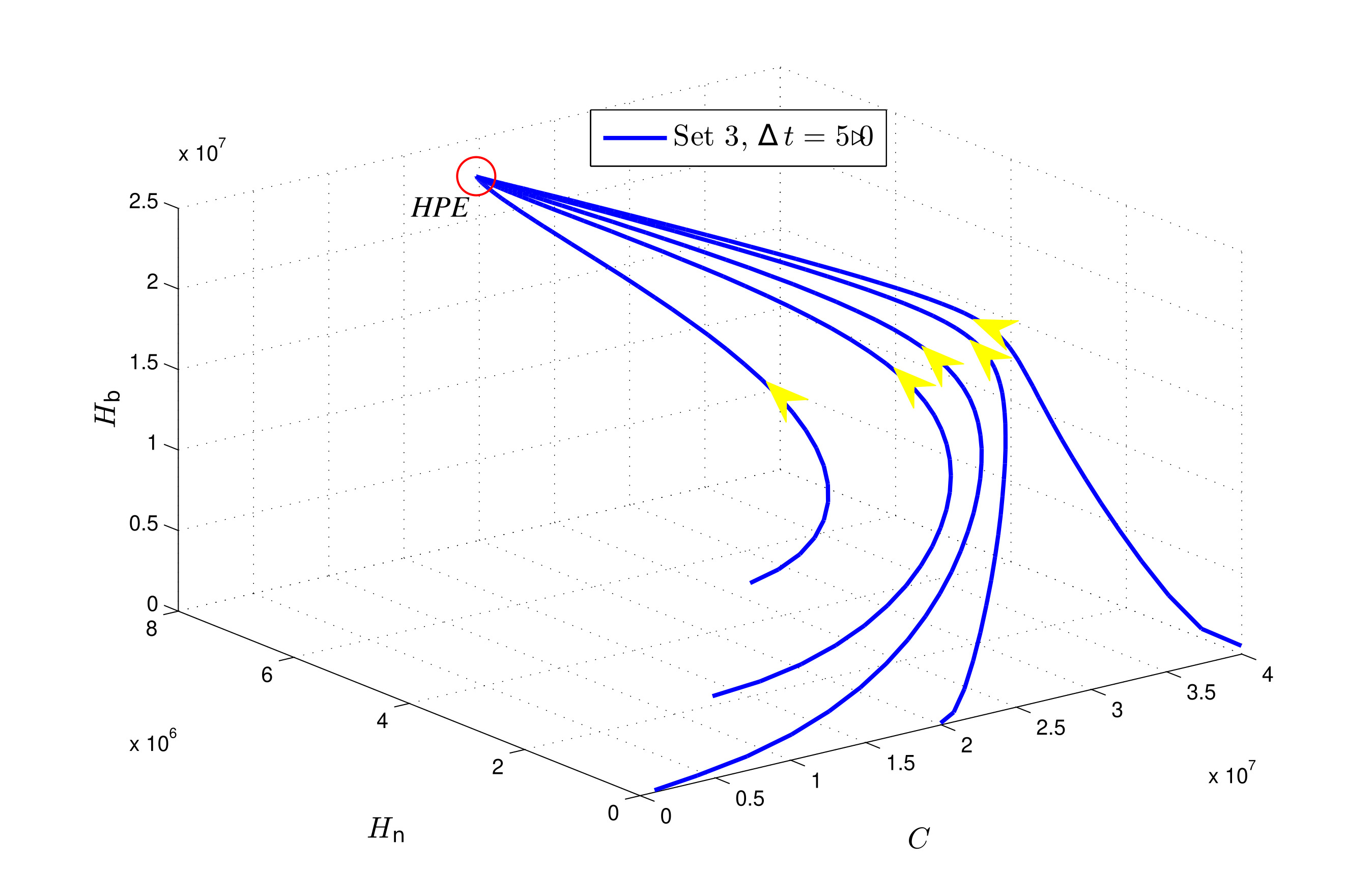}
\label{Figure:3a}
}\hfill
\subfloat[$\Delta t = 1.0$]{%
\includegraphics[height=10cm,width=9cm]{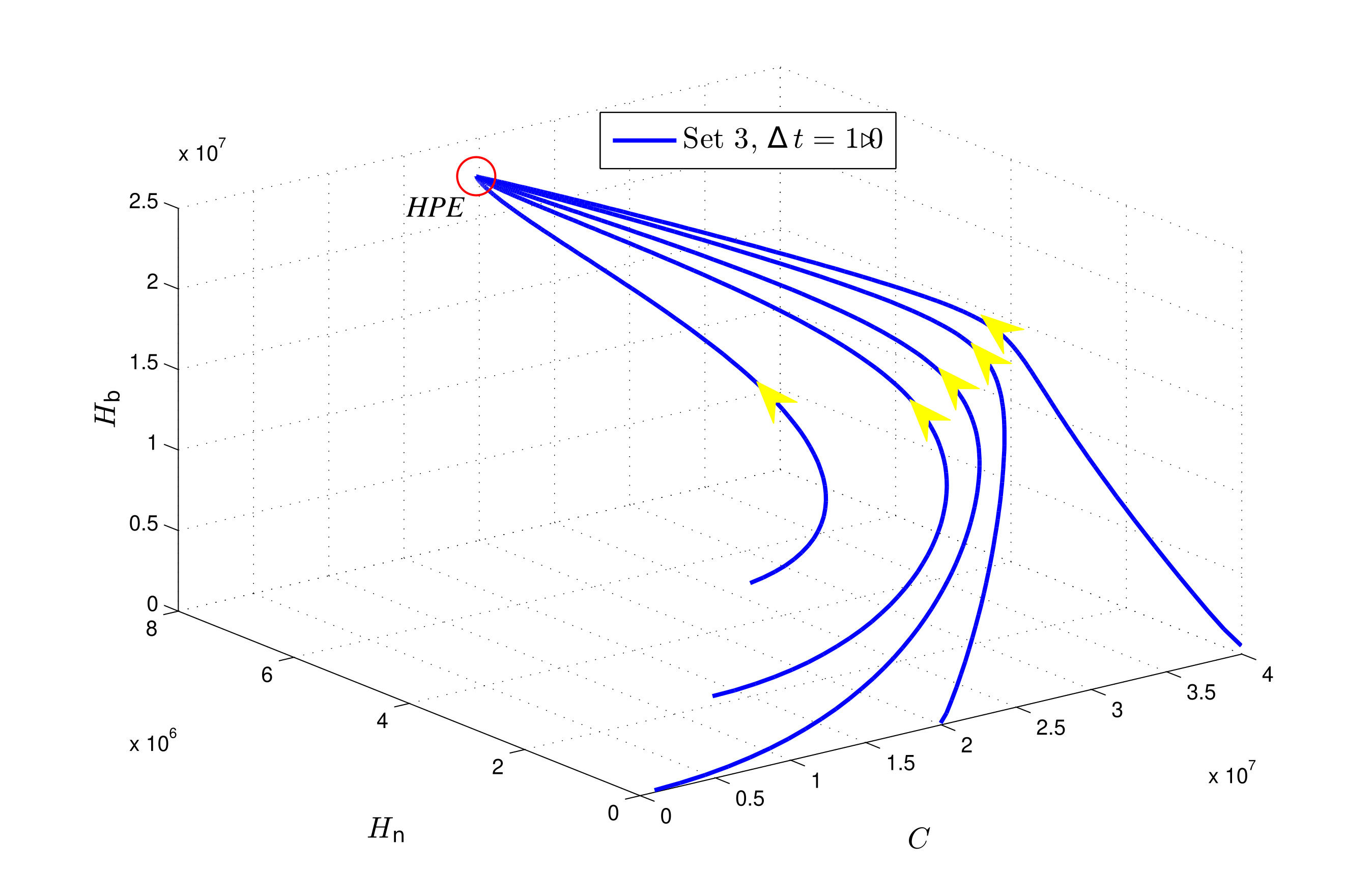}
\label{Figure:3b}
}\hfill
\subfloat[$\Delta t = 10^{-3}$]{%
\includegraphics[height=10cm,width=16cm]{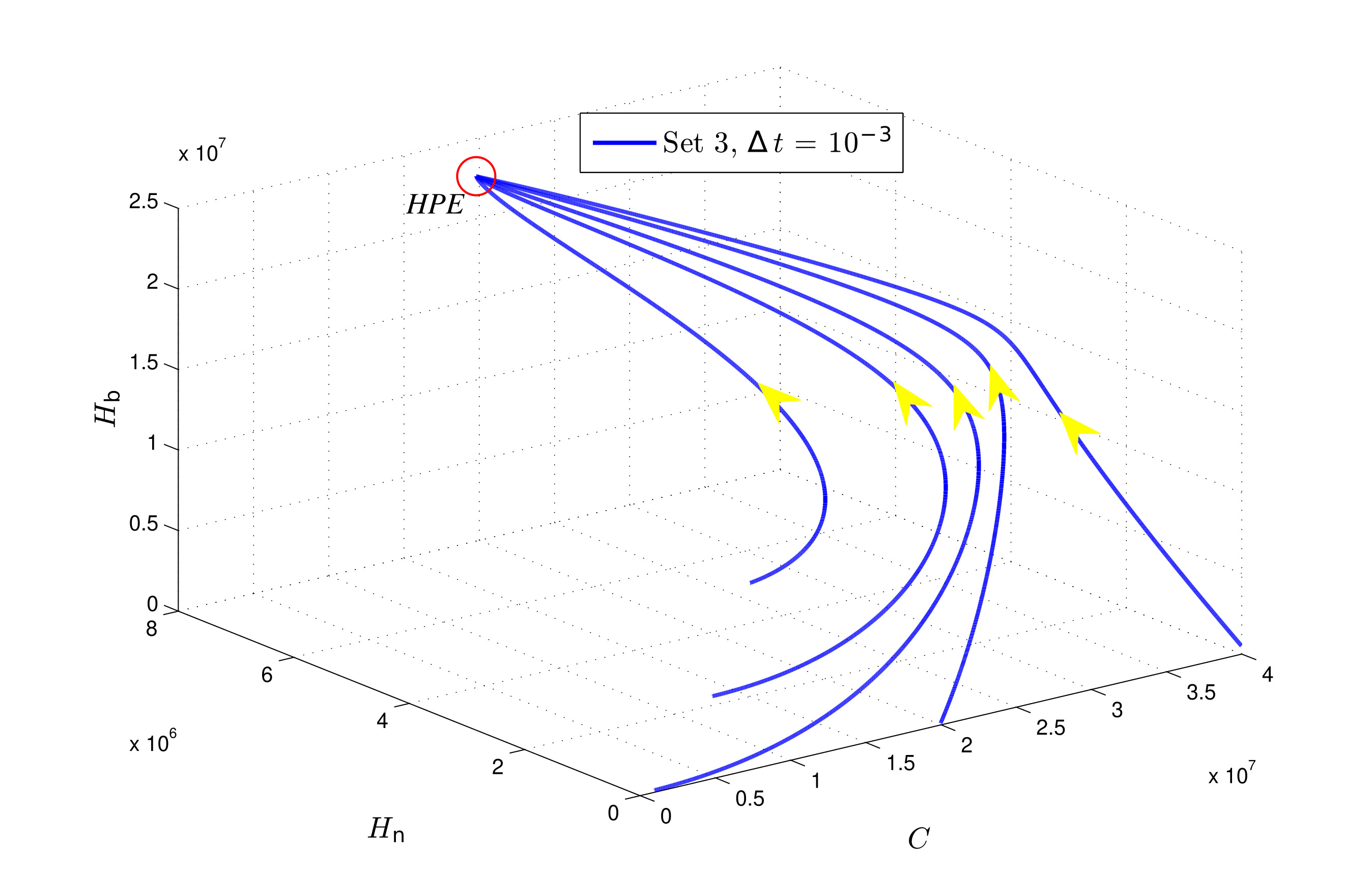}
\label{Figure:3c}
}
\caption{The numerical solutions  generated  by the second-order positivity-preserving NSFD scheme \eqref{eq:NSFD2} using parameter set 3 in Table \ref{Table5}.}
\label{Fig:3}
\end{figure}

\medskip\noindent
\textbf{Ethical Approval:} Not applicable.\\
\textbf{Availability of supporting data:} The data supporting the findings of this study are available within the article [and/or] its supplementary materials.\\
\textbf{Conflicts of Interest:} The author declares no conflicts of interest to disclose.\\
\textbf{Authors' contributions:} \textbf{Manh Tuan Hoang and Matthias Ehrhardt:} 
Writing review \& editing, Writing original draft, Visualization, Validation, Supervision, Software, Resources, Project administration, Methodology, Investigation,
Formal analysis, Data curation, Conceptualization, Funding acquisition.\\
\textbf{Hoai Thu Pham:}
Writing review \& editing, Writing original draft, Visualization, Validation, Software, Resources, Methodology, Investigation,
Formal analysis, Data curation, Conceptualization.\\
\textbf{Funding information:} Not available.


\bibliographystyle{amsalpha}

\begin{thebibliography}{99}

\bibitem{Allen}
L. J. S. Allen, 
An Introduction to Mathematical Biology, Prentice Hall,  2007.

\bibitem{Angstmann}
C. N. Angstmann, A. M. Erickson, B. I. Henry, A. V. McGann, J. M. Murray,  A. James,
Fractional Order Compartment Models, 
SIAM J. Appl. Math. 77 (2017), 430--446.

\bibitem{Ascher}
U. M. Ascher, L. R. Petzold,
Computer Methods for Ordinary Differential Equations and Differential-Algebraic Equations, 
SIAM, Philadelphia, 1998.

\bibitem{Beddington}
J. R. Beddington, 
Mutual Interference Between Parasites or Predators and its Effect on Searching Efficiency, 
J. Animal Ecol. 44 (1975), 331--340.

\bibitem{Bhattacharya}
S. Bhattacharya, K. Gaurav, S. Ghosh, 
Viral marketing on social networks: An epidemiological perspective, 
Physica A 525 (2019), 478--490.

\bibitem{Brauer}
F. Brauer, Compartmental Models in Epidemiology, 
In: Brauer, F., van den Driessche, P., Wu, J. (eds.), 
Mathematical Epidemiology. Lecture Notes in Mathematics 1945. 
Springer, Berlin, Heidelberg.

\bibitem{Capasso}
V. Capasso, G. Serio, 
A generalization of the Kermack-McKendrick deterministic epidemic model, 
Math. Biosci. 42 (1978), 43--61.

\bibitem{Castillo-Chavez}
C. Castillo-Chavez, Z. Feng, W. Huang, On the computation of $R_0$ and its role in global stability, 
in: C. Castillo-Chavez, S. Blower, P. van den Driessche, D. Kirschner, A.-A. Yakubu (Eds.), 
Mathematical Approaches for Emerging and Reemerging Infectious Diseases: An Introduction, Springer, 2002, p. 229.

\bibitem{ConnellMcCluskey}
C. Connell McCluskey, Y. Yang, 
Global stability of a diffusive virus dynamics model with general incidence function and time delay, 
Nonlin. Anal. Real World Appl. 25 (2015), 64--78.

\bibitem{Copper}
G. J. Cooper, J. H. Verner, 
Some Explicit Runge-Kutta Methods of High Order, 
SIAM J. Numer. Anal. 9 (1972), 389--405.

\bibitem{Cresson}
J. Cresson, F. Pierret, 
Non standard finite difference scheme preserving dynamical properties, 
J. Comput. Appl. Math. 303 (2016), 15--30.

\bibitem{Daley}
D. J. Daley, D. G. Kendall, 
Epidemics and rumours,
Nature 204(4963) (1964), 1118.

\bibitem{DeAngelis}
D. L. DeAngelis, R. A. Goldstein, R. V. O'Neill, 
A Model for Tropic Interaction, 
Ecology 56 (1975), 881--892.

\bibitem{delRey}
 A. M. del Rey, 
 Mathematical modeling of the propagation of malware: a review, 
 Secur. Commun. Netw. 8 (2015), 2561--2579.
 
 \bibitem{Dietz}
K. Dietz, 
Epidemics and Rumours: A Survey, 
J. Royal Stat. Soc. Ser. A (General) 130(4) (1967), 505--52.

\bibitem{Dimitrov3}
D. T. Dimitrov and H. V. Kojouharov, {Dynamically consistent numerical methods for general productive-destructive systems}, 
J. Differ. Equ. Appl. {17} (2011) 1721--1736.

\bibitem{Ding}
Y. Ding, L. Zhu, 
Turing instability analysis of a rumor propagation model with time delay on non-network and complex networks, 
Inform. Sci.  667 (2024), 120402.
\bibitem{Goffman}
W. Goffman, V. A. Newill, 
Generalization of Epidemic Theory: An Application to the Transmission of Ideas, 
Nature 204 (1964), 225--228.

\bibitem{Hassouni}
H. Hassouni, A. E. Bhih, O. Balatif, 
Mathematical modeling, analysis, and optimal control of hacker behavior in cybersecurity systems, 
J. Appl. Math. Comput. 72(2026), 32.

\bibitem{HoangMatthias}
M. T. Hoang,  M. Ehrhardt, 
Differential equation models for infectious diseases: Mathematical modeling, qualitative analysis, numerical methods and applications, 
SeMA (2025). 

\bibitem{Hoang2023}
M. T. Hoang, 
Dynamical analysis of a generalized hepatitis B epidemic model and its dynamically consistent discrete model, 
Math. Comput. Simul. 205 (2023), 291--314.

\bibitem{HoangMatthias2026}
M. T. Hoang, M. Ehrhardt, 
A generalized second-order positivity-preserving numerical method for non-autonomous dynamical systems with applications,
Appl. Math. Comput. 524 (2026), 130029.

\bibitem{Hoang2026}
M. T. Hoang, 
Mathematical analysis and numerical simulation of a generalized epidemiological model for malware propagation, 
Nonlin. Dynam. 114 (2026), 53.

\bibitem{Horvath}
Z. Horv\'ath, 
On the positivity step size threshold of Runge-Kutta methods,
Appl. Numer. Math. 53 (2005), 341--356.

\bibitem{Kawachi}
K. Kawachi, 
Deterministic models for rumor transmission, 
Nonlin. Anal. Real World Appl. 9 (2008), 1989--2028.

\bibitem{Kermack-McKendrick}
W. O. Kermack, A. G. McKendrick, 
A contribution to the mathematical theory of epidemics, 
Proc. Royal Soc. London - Ser. A 115 (1927), 700--721.

\bibitem{Kermack-McKendrick1}
W. O. Kermack, A. G. McKendrick, 
Contributions to the mathematical theory of epidemics. II. -The problem of endemicity,
Proc. Royal Soc. London - Ser.  A 138 (1932), 55--83.

\bibitem{Kermack-McKendrick2}
W. O. Kermack, A. G. McKendrick,
Contributions to the mathematical theory of epidemics. III. -Further studies of the problem of endemicity, 
Proc. Royal Soc. London - Ser. A 141 (1933), 94--122. 

\bibitem{Khalil}
H. K. Khalil,
Nonlinear systems, 
Third Edition, Prentice Hall, 2002.

\bibitem{LiTeng}
J. Li, Z. Teng, 
Hopf bifurcation and stability for an SIRS epidemic model incorporating immune waning, logistic growth and saturated treatment, 
Math. Comput. Simul. 247 (2026), 758--784.

\bibitem{Martcheva}
M. Martcheva, 
An Introduction to Mathematical Epidemiology,
Springer New York, NY, 2015.

\bibitem{McNabb} 
A. McNabb, 
Comparison theorems for differential equations, 
J. Math. Anal. Appl. 119 (1986), 417--428.

\bibitem{Mickens1}
R. E. Mickens,
Nonstandard Finite Difference Models of Differential Equations, 
World Scientific, Singapore, 1994.

\bibitem{Mickens2}
R. E. Mickens,
Applications of Nonstandard Finite Difference Schemes,  
World Scientific, Singapore, 2000.

\bibitem{Mickens3}
R. E. Mickens,
Dynamic consistency: a fundamental principle for constructing nonstandard finite difference schemes for differential equations,  
J. Differ. Eqs. Appl. 11 (2005), 645--653.
%
%
\bibitem{Mickens4}
R. E. Mickens,  
Advances in the Applications of Nonstandard Finite Difference Schemes, 
World Scientific, Singapore, 2005.
%

\bibitem{Mickens5}
R. E. Mickens, 
Nonstandard Finite Difference Schemes: Methodology and Applications, 
World Scientific, 2020.

\bibitem{Murray}
W. H. Murray, 
The application of epidemiology to computer viruses,
Comput. Secur. 7 (1988), 139--145.


\bibitem{Patidar1}
K. C. Patidar, 
On the use of nonstandard finite difference methods,
J. Differ. Eqs. Appl. 11 (2005), 735--758.

\bibitem{Patidar2}
K. C. Patidar, 
Nonstandard finite difference methods: Recent trends and further developments, 
J. Differ. Eqs. Appl. 22 (2016), 817--849.

\bibitem{Piqueira1}
J. R. C. Piqueira, O. V, Araujo, A modified epidemiological model for computer viruses,
Appl. Math. Comput. 213 (2009), 355--360.

\bibitem{Piqueira2} 
J. R. C. Piqueira, A. A. de Vasconcelos, C. E. C. J. Gabriel, V. O. Araujo, Dynamic models for computer viruses,
Comput. Secur. 27 (2008), 355-359.

\bibitem{Piqueira3}
 J. R. C. Piqueira, M. A. M. Cabrera, C. M. Batistela, 
 Malware propagation in clustered computer networks, 
 Physica A 573 (2021), 125958.
 
\bibitem{Piqueira4}
J. R. C. Piqueira, M. Zilbovicius, C. M. Batistela, 
Daley-Kendal models in fake-news scenario, 
Physica A 548 (2020), 123406.

\bibitem{Seibert}
P. Seibert, R. Suarez, 
Global stabilization of nonlinear cascade systems, 
Syst. Contr. Lett. 14 (1990), 347--352.

\bibitem{Smith}
H. L. Smith and P. Waltman,  
The Theory of the Chemostat: Dynamics of Microbial Competition,
Cambridge University Press, 2009.

\bibitem{Stuart}
A. Stuart, A. R. Humphries, 
Dynamical systems and numerical analysis,
Cambridge University Press, 1998.

\bibitem{Sun}
H. Sun, J. Wang, 
Dynamics of a diffusive virus model with general incidence function, cell-to-cell transmission and time delay, 
Comput. Math. Appl. 77 (2019), 284--301.

\bibitem{Tian}
Y. Tian, X. Liu,
Global dynamics of a virus dynamical model with general incidence rate and cure rate, 
Nonlin. Anal. Real World Appl. 16 (2014), 17--26.

\bibitem{vdDriessche}
P. van den Driessche, J. Watmough, 
Reproduction numbers and sub-threshold endemic equilibria for compartmental models of disease transmission, 
Math. Biosci. 180 (2002), 29--48.

\bibitem{Wood0}
D. T. Wood, D. T. Dimitrov, and H. V. Kojouharov, 
A nonstandard finite difference method for $n$-dimensional productive-destructive systems, 
J. Differ. Equ. Appl. 21 (2015), 240--254.

\bibitem{Wood1}
D. T. Wood and H. V. Kojouharov, 
A class of nonstandard numerical methods for autonomous dynamical systems, 	
Appl. Math. Lett. {50} (2015), 78--82.

\bibitem{Xu}
R. Xu, Z. Ma, 
An HBV model with diffusion and time delay, 
J. Theor. Biol. 257 (2009), 499--509.

\bibitem{Zhang}
Y. Zhang, Z. Xu, 
Dynamics of a diffusive HBV model with delayed Beddington-DeAngelis response, 
Nonlin. Anal. Real World Appl. 15 (2014), 118--139.

\bibitem{Zhu}
Q. Zhu, G. Zhang, X. Luo, C. Gan, 
An industrial virus propagation model based on SCADA system, 
Inform. Sci. 630 (2023), 546--566.

\end{thebibliography}

\end{document}